\documentclass[12pt]{amsart}
\usepackage{amssymb} 

\newtheorem{thm}{Theorem}[section]
\newtheorem{cor}[thm]{Corollary}
\newtheorem{lem}[thm]{Lemma}

\theoremstyle{definition}

\newtheorem{rem}[thm]{Remark}

\newtheorem{que}[thm]{Question}

\numberwithin{equation}{section}
 
\usepackage[colorlinks]{hyperref}

\begin{document}
\title[On a  bijection between   a finite group\ldots    ]
{On a  bijection between a     finite group to a non-cyclic group with   divisibility of element orders}%
\author[    ]{  Mohsen Amiri}%
\address{ Departamento de  Matem\'{a}tica, Universidade Federal do Amazonas.}%
\email{m.amiri77@gmail.com}%
\email{}
\thanks{}
\thanks{}


\begin{abstract}
Consider a finite group $G$ of order $n$ with a prime divisor $p$. In this article, we establish, among other results, that if the Sylow $p$-subgroup of $G$ is neither cyclic nor generalized quaternion, then there exists a bijection $f$ from $G$ onto the abelian group $C_{\frac{n}{p}}\times C_p$ such that for every element $x$ in $G$, the order of $x$ divides the order of $f(x)$. This resolves Question 1.5 posed in \cite{mohsen}.
As application of our results, we show that  the group  with the third  largest value of the sum of element orders in the set of all finite groups of order $n$ is a solvable $p$-nilpotent group where $p$ is the smallest prime divisor of  $n$ such that the Sylow $p$-subgroups are not cyclic.
\end{abstract}

\maketitle


\section{\bf Introduction}

For a group $G$, the set of orders of all elements in $G$ is denoted by $\omega(G)$, recently referred to as the spectrum of $G$. One of the intriguing concepts in Finite Group Theory that has garnered significant attention from researchers is the characterization of finite groups based on their element orders.

To illustrate, a finite group $G$ is considered recognizable by its spectrum if every finite group $H$ with an identical spectrum as $G$ is isomorphic to $G$. In 1987, W. Shi, in a letter to J. Thompson, proposed a conjecture asserting that every finite simple group is recognizable in the class of finite groups by both its order and the set of element orders. In response, J. Thompson commended this conjecture and introduced another one: every finite non-abelian simple group is recognizable in the class of finite groups with a trivial center, identified by the set of sizes of conjugacy classes. After extensive efforts by numerous mathematicians, initiated in \cite{97} and \cite{18}, the validity of both conjectures was ultimately established.
 In \cite{lad}, I. M. Isaacs poses a question about the spectrum of a finite group $G$ from a distinct perspective:

\begin{que}\label{is}(see Problem 18.1 in  \cite{Kh})
{\it Does there necessarily exist a bijection $f$ from $G$ onto a cyclic group of order $n$ such that for each element $x\in G$, the order of $x$ divides the order of $f(x)$}?
\end{que}

An affirmative answer
has been given by Ladisch in the special case when $G$ is a solvable group. 
Recently, author has given an   affirmative answer to Question \ref{is} in \cite{mohsen2}.
Note that if $G$ is a non-cyclic group, then  we can not  replace the cyclic group $C_n$ with other abelian groups  in Question \ref{is}. For example,  in \cite{mohsen}, authors show that if $n=8m$, then 
there is no any bijection $f$ from
$C_m\times Q_8$ onto  $C_{4m}\times C_2$ such that $o(x)\mid o(f(x))$. Hence, the cyclic group $C_n$  cannot be replaced by other non-cyclic abelian  groups of order $n$. Suppose that $H$ is a finite group. Then  $\mathfrak{cl}(H)$ denotes the set of all non-isomorphic finite groups $G$  for which there is a bijection $f:G\longrightarrow H$ such that for each element $x\in G$, the order of $x$ divides the order of $f(x)$.
 Clearly, by the affirmative answer to Question \ref{is}, we have   $\mathfrak{cl}(C_n)$  is the set of all finite groups of order $n$.   It is a natural question what we can say about the elements in $\mathfrak{cl}(H)$ if we know $H$.  The main objective of this article is   to give an affirmative answer to following question: 

 \begin{que}( Question 1.5 of \cite{mohsen})\label{iss}
Let $G$  be a finite group of order $n$ and $p$ be a prime divisor of $n$. If a Sylow $p-$subgroup of $G$ is neither cyclic nor generalized quaternion, then is it true that $G\in \mathfrak{cl}(C_{\frac{n}{p}} \times C_p)$?
\end{que}
More precisely, we prove the following Theorem:
\begin{thm}\label{fmain}
    
Let $G$ be a  finite  group of order $n=p_1^{\alpha_1}\ldots p_k^{\alpha_k}$ where $p_1<\ldots<p_k$ are prime numbers, and let $P_i\in Syl_{p_i}(G)$ for all $i=1,\ldots,k$. Then $G\in \mathfrak{cl}(\Pi_{i=1}^k Q_i)$ where     for $i=1,2,\ldots,k$, we have
$Q_i=C_{p_i^{\alpha_1}}$ if    $P_i$ is cyclic and $Q_i=C_{p_i^{\alpha_i}}\times C_{p_i}$ whenever $P_i$ is not cyclic and either    $p_i>2$ or $p_i=2$ and $|\{x\in G,x^{n_{p_i'}p_i}=1\})|>2n_{2'}$.
If $P_1$ is not cyclic, and  $|\{x\in G,x^{n_{p_i'}p_i}=1\}|=2n_{2'}$, then  $Q_1$ is  generaized quaternion group of order $2^{\alpha_1}.$ 
\end{thm}

Note that our proof is elementary, and we do not rely on the classification of finite simple groups in our arguments.

For any two positive integers $n$  and  $p$, the biggest divisor of $n$ which is co-prime to $p$ is denote by $n_{p'}$ and $n/n_{p'}$ is denoted by $n_p$.
Let $G$ be a finite group of order $n$ and $U$ a subset of $G$.
We denote by $Sol(U,d,G)$ to be the set of all $x\in G$ such that $x^m\in U^G=\{u^g: b\in U, g\in G\}$ for some divisor $m$ of $d$.
Also, For any element $x$ of a group $G$, the symbol  $[x]$ denote  the set of all  generators of the cyclic subgroup $\langle x\rangle$ of $G$.

As the first consequence of our result, we obtain the following theorem:
\begin{thm}\label{cyc}
    Let $G$  be a finite group of order $n$ such that for any prime divisor $p$ of $n$, if $p^2\mid n$, then $G\not\in \mathfrak{cl}(C_{\frac{n}{p}} \times C_p)$.
    Then $G$ is a solvable group and all Sylow subgroups of $G$ are cyclic or generalized Quaternion.
\end{thm}
 
 In \cite{jaf}, H. Amiri, S.M.J. Amiri and  M. Isaacs studied the spectrum of a finite group $G$ from a different point of view: they introduced the function  $\psi(G)=\sum_{g\in G}o(g)=\sum_{m\in \omega(G)}m\cdot s(m)$, and they proved 
  that $\psi(C_n)$ is maximal in the set $\{\psi(G) :
|G|=n\}$. 
  Later, other authors, in \cite{2}, and independently, R. Shen, G. Chen, and C. Wu, in \cite{15}, investigated groups with the second-largest value of the sum of element orders. The function $\psi$ has been considered in various works, as seen in \cite{1}, \cite{2}, \cite{6}, \cite{7}, \cite{9}, \cite{10}, \cite{11}, \cite{16}, \cite{Marefat2}. While certain papers focused on identifying the maximum, second-largest, and minimum values of $\psi$, others endeavored to formulate novel criteria for the structural properties (including solvability, nilpotency, etc.) of finite groups. Also, recently, H. Kishore Dey and A. Mondal, in \cite{Kishore}, found an exact upper bound for the sum of powers of element orders in non-cyclic finite group.  Two finite groups $G$ and $H$ are called the same order type groups whenever $|Sol(1,d,G)|=|Sol(1,d,H)|$ for all positive integers $d$.  For two finite group $G$ and $H$ of the same order $n$ we show that  if $G\in \mathfrak{cl}(H)$, then $\psi(G)=\psi(H)$ if and only if $G$ and $H$ are the same order type groups. As application of our results, we show that  the group  with the third  largest value of the sum of element orders in the set of all finite groups of order $n$ is a solvable $p$-nilpotent group where $p$ is the smallest prime divisor of order of $n$ such that the Sylow $p$-subgroups are not cyclic. More precisely, We prove the following theorem.
  \begin{thm}\label{main5}
    Let $G$ be a finite group of order $n$ with  the third-largest value of $\psi$ within the set of all finite groups of order $n$. Denote by $P$ a Sylow $p$-subgroup of $G$, where $p$ is the smallest prime divisor of $n$, and $P$ is non-cyclic. Consequently, $G$ can be expressed as a semidirect product $G=K\rtimes P$, where $K$ is a Hall $p'$-subgroup of $G$. Furthermore, if all Sylow subgroups of $G$ except the Sylow $p$-subgroups are cyclic, then $K$ is cyclic.
Moreover,

(i) If $P$ contains an abelian subgroup of rank $3$, then $G\cong C_{n/|P|}\times P$. Furthermore, $P$ and $C_{|P|/p^2}\times (C_p)^2$ are groups of  the same order type.

(ii) If there exist non-cyclic subgroups $R\in Syl_r(G)$ and $Q\in Syl_q(G)$ where $r<q$, then $G\cong C_{\frac{n}{|Q||R|}} \times R \times Q$. Additionally, $RQ$ and $C_{\frac{|R||Q|}{r^2q^2}} \times (C_q)^2 \times (C_r)^2$ are groups of the same order type.
 \end{thm}

  Also,  we obtain the following corollary about the structure of finite group of order $n$ with the second-largest value of the sum of $k$-th powers of element orders:
 \begin{cor} \label{coo}
    Let $G$ be a finite group of order $n$ with the second-largest value of the sum of $k$-th powers of element orders. Then, either $G \cong C_{n/p} \times C_p$ for some prime divisor $p$ of $n$, or $G \cong C_m \times Q_{2^e}$ where $m$ is an odd number such that $n = m2^e$, or all Sylow subgroups of $G$ are cyclic.
 \end{cor}

If $G$ is a finite group, let $o(G)  = \frac{1}{|G|}\sum_{g\in G}o(g)$.
 It is easy to see that $ h(G)\geq o(G)$, where $h(G)$ is the number of conjugacy classes of $G$.
 It is known by Lindsey's Theorem \cite{Lind} that  $o(G)\leq o(C_n)$ for any finite group $G$ of order $n$.
 If in Theorem \ref{sec}, set $l=1$ we obtain the following theorem:

\begin{thm}\label{oo}
  Let $G$ be a  finite  group of order $n=p_1^{\alpha_1}p_2^{\alpha_2}\ldots p_k^{\alpha_k}$ where $p_1<\ldots <p_k$ are prime numbers. 
Let $N_i$ be an elementary abelian     $p_i$-subgroup of $G$ of order $p_i^{r_i}\leq p_i^{p_i-1}$   for all $i = 1, 2, \ldots, k$.    Then  
        $$ o(G)\leq   \Pi_{i=1}^k((\frac{1-p_i}{p_i^{\alpha_i}}+p_i^{r_i-\alpha_i}( \frac{p_i^{2(\alpha_i+1)}-1}{p_i+1})).$$

\end{thm} 
 
We denote the set of all elements in any subset $U$ of group $G$ with order $d$ as $B_d(U)$. Additionally, we represent the unique cyclic subgroup of order $d$ within the cyclic group of order $n$ by $C_{d}$.    
In what follows, we adopt the notations  established in the Isaacs' book on finite groups  \cite{I}.


\section{Generalizations of the Frobenius Theorem}\label{sec2}

 We start with the following  
fundamental theorem proved by Frobenius \cite{fr}, more than hundred years ago, in 1895:

\begin{thm}
If $d$ is a divisor of the order of a finite group $G$, then the number of solutions of
$x^d=1$ in $G$ is a multiple of $d$.
\end{thm}
 
Let $G$ be a finite  group of order $n$ and $p$ be a prime divisor of $n$. We denote the set of all $p$-elements $x\in G$ such that 
$o(hy)\mid lcm(o(h),o(y))$ for any $p$-element $y$ of $G$ and all $h\in \langle x\rangle$ by $LCM_p(G)$.
From Lemma 2.3 of \cite{man2}, $LC_p(G):=\langle LCM_p(G)\rangle$ is a characteristic nilpotent subgroup of $G$.  The $p$-subgroup $N$ of $G$ of order $p^r$ and exponent $p^s$ is called $A$-regular(Almost regular) whenever  for any $y\in Sol(1,|P|,G)$ and $M\leq N$, if $o(y)>p^s$ and $M\langle y\rangle \leq G$, then  $LCM_p(\langle y\rangle M)=\langle y\rangle M$.
To prove our second main result we need the following  generalizations of the    Frobenius'  theorem.

 \begin{thm}
 \label{divv22}

  Let $G$ be a  finite  group of order $n$, and let $P\in Syl_p(G)$.
    Suppose that $P$ has an abelian      $A$-regular subgroup $N$ of order $p^r$ and exponent $p^s$.
    Let $d$ be a divisor of $n_{p'}$ and $p^s\leq p^j\mid exp(P)$.
Then $dp^{r+j-s}\mid |Sol(1,dp^j,G)|$.
 
 \end{thm}
 \begin{proof}
We proceed by induction on $d$. First, suppose that $d=1$.
Now, we proceed by induction on $j$. Let $exp(P)=p^t$.
If $p^t=p^j$, then by the Frobenius' theorem, $|P|\mid |Sol(1,|P|,G)|=|Sol(1,p^t,G)| $, so  $p^{r+j-s}\mid |Sol(1,p^t,G)|$. Hence, we may assume that  $t>j$.
By the induction hypothesis, $p^{r+j+1-s} \mid |Sol(1,p^{j+1},G)|$. We have $Sol(1,p^{j+1},G)\setminus Sol(1,p^{j},G)=B_{p^{j+1}}(G)$. 
Let $U=B_{p^{j+1}}(G)$, and let $M$ be a   subgroup of $N$ of order $p^{r_M}$ and of  exponent $p^{s_M}$, and let $$V_M=\{b\in U:\langle b\rangle M\leq G, \ N_N(\langle b\rangle M)=M\}.$$
Furthermore, since $N$ is abelian, for any $x\in N$ and $b\in V_M$, we have $(\langle b\rangle M)^x)=\langle b^x\rangle M.$
If $z\in N_N(\langle b^x\rangle M)$, then 
$z\in N_N(\langle b\rangle M)$, so $z\in M$,
consequently, $b^x\in V_M$.
Let $u\in B_{p^{j+1}}(M\langle b\rangle)$.
Then $M\langle b\rangle=M\langle u\rangle$, so $u\in V_M$, also,  $M\langle b\rangle\cap M\langle h\rangle\leq  M\langle b^2\rangle$ for any $h\in B_{2^{j+1}}(V_M)\setminus B_{2^{j+1}}(M\langle b\rangle).$

If $G= M\langle b\rangle$, then by our assumption 
$exp(\Omega_1(M\langle b\rangle))=p^j$.
It follows that 
$$p^{r+j-s}\mid  |B_{p^j}(\Omega_1(G)|=|Sol(1,p^j,G)|.$$

So  $G\neq M\langle b\rangle$.  
By the induction hypothesis, $$p^{r_M+j-s_M}\mid |\Omega_j(M\langle b\rangle)|=|Sol(1,p^j,M\langle b\rangle)|.$$ 
Since $exp(\Omega_j(M\langle b\rangle))=p^j$ and $p^{r_M+j-s_M}\mid |\Omega_j(M\langle b\rangle)|$, we have 
$$p^{r_M+j-s_M}\mid |B_{p^{j+1}}(M\langle b\rangle)|=|M\langle b\rangle|-|\Omega_j(M\langle b\rangle)|.$$
Clearly,  $[N:N_N(M\langle b\rangle)]=[N:M].$
Since $p^{r-r_M-(s-s_M)}\mid [N:M]$, we deduce that  
$$p^{r+j-s}\mid [N:M]|B_{p^{j+1}}(M\langle b\rangle)| =|B_{p^{j+1}}(M\langle b^N\rangle)|.$$

 Additionally, since $V_M$ is a union of  disjoint subsets $$B_{p^{j+1}}(M\langle u_{1}^N\rangle),\ldots,B_{p^{j+1}}(M\langle u_{q}^N\rangle),$$ 
 we conclude that  $p^{r+j-s}\mid |V_M|$. 
Let $S$ denote the set of all subgroups of $N$. Consider $M,K \in S$ such that $M \neq K$, and let $a \in V_M \cap V_K$. Then, both $\langle a \rangle K$ and $\langle a \rangle M$ are subgroups of $G$. Since $N$ is an abelian group, the subgroup $H:=\langle a \rangle KM$ is a subgroup of $G$. Observing that $\langle a \rangle K$ is normal in $H$, we deduce $KM \leq N_N(\langle a \rangle K)=K$, which contradicts the assumption. Therefore, $V_M \cap V_K = \varnothing$.

Considering $B_{p^{j+1}}(G)=\bigcup_{M \in S}V_M$, and noting that $p^{r+j-s}$ divides $|V_M|$ for any $M \in S$, we infer that $p^{r+j-s}$ divides $|B_{p^{j+1}}(G)|$. Consequently,

$$p^{r+j-s}\mid |Sol(1,p^{j+1},G)|-|B_{p^{j+1}}(G)|=|Sol(1,p^j,G)|.$$

So $d>1$.
Let $q$ be a prime divisor of $d$, and let 
$gcd(d_q,exp(G))=q^t$.
Then 
$$|Sol(1,dp^j,G)|=\sum_{i=0}^t|Sol(B_{q^i}(G),d_{q'}p^j,G)|.$$
By the induction hypothesis,
  $p^{r+j-s}\mid |Sol(1,d_{q'}p^j,G)|$. 
Let $1\leq i\leq t$, and let $y\in B_{q^i}(G)$.
By Proposition 2.3 of \cite{mohsen2}, $$Sol([y],d_{q'p},G)|=|y^G||Sol([y],d_yp^j,C_G(y))|$$ where $d_y=gcd(|C_G(y)|,d_{q'})$.
Let $Q\in Syl_p(C_G(y))$. We may assume that $Q\leq P$. Let 
$p^{r_y}=|N\cap Q|$, $exp(N\cap Q)=p^{s_y}$ and $gcd([Q:N\cap Q],p^j)=p^{j_y}$. 
By the  induction hypothesis, 
 
$$p^{r_y+j_y-s_y}\mid |Sol([y],d_yp^{j_y},C_G(y))|=|Sol(1,d_yp^{j_y},\frac{C_G(y)}{\langle y\rangle})|.$$
 
Since $ p^{r-r_y+j-j_y-(s-s_y)}\mid |y^G|$, we deduce that 
$p^{r+j-s}\mid Sol([y],d_yp^{j_y},G)|$.
Consequently,
$p^{r+j-s}\mid |Sol(B_{q^i}(G),d_{q'}p,G)|$. By Proposition 2.3 of \cite{mohsen2}, $d\mid |Sol(1,dp^j,G)|$,  so 
$dp^{r+j-s}\mid |Sol(1,dp^j,G)|$.

 \end{proof}
 \begin{rem}\label{rem1}
  Note that in Theorem \ref{divv22}, if $N$ is a Dedekind $p$-subgroup, the theorem remains true with a slight modification in its proof.
 \end{rem}
More generally, in the following theorems we generalize   Proposition 2.3 of \cite{mohsen2} and Theorem \ref{divv22}.
 \begin{thm}
 \label{divv2}
    Let $G$ be a finite group of order $n$, and let $P\in Syl_p(G)$. 
    Suppose that $P$ contains a Dedekind  $A$-regular $p$-subgroup $N$ of order $p^r$ and exponent $p^s$.
Let $d$ be a divisor of $n_{p'}$, and suppose $p^s \leq p^j \mid exp(P)$.
For any element $y \in G$ whose order is co-prime to $dp$, the cardinality $|Sol([y], dp^j, G)|$ is a multiple of $  lcm(\varphi(o(y)), p^{r+j-s} d_{o(y)'})$.

 \end{thm}
 \begin{proof}
    We proceed by induction on the order of the group $|G|$.
     If $y=1$, the result follows from Theorem $\ref{divv22}$ and \ref{rem1}.
So, let's consider the case where $y\neq 1$.
     Then $$Sol([y],d_{q'p},G)|=|y^G||Sol([y],d_yp^j,C_G(y))|$$ where $d_y=gcd(|C_G(y)|,d_{q'})$.
Let $Q\in Syl_p(C_G(y))$. We may assume that $Q\leq P$. Let 
$p^{r_y}=|N\cap Q|$, $exp(N\cap Q)=p^{s_y}$  and $gcd([Q:N\cap Q],p^j)=p^{j_y}$. 
By the  induction hypothesis, 
$$p^{r_y+j_y-s_y}\mid |Sol([y],d_yp^{j_y},C_G(y))|=|Sol(1,d_yp^{j_y},\frac{C_G(y)}{\langle y\rangle})|.$$
Since $ p^{r-r_y+j-j_y-(s-s_y)}\mid |y^G|$, we deduce that 
$p^{r+j-s}\mid Sol([y],d_yp^{j_y},G)|$.
Consequently, 
$p^{r+j-s}\mid | Sol([y],d_{q'}p^j,G)|$.
From Proposition 2.3 of \cite{mohsen2}, $lcm(d_{o(y)'}, \varphi(o(y))\mid | Sol([y],d_{q'}p^j,G)|$.
Consequently, 
$$lcm(\varphi(o(y)), p^{r+j-s}d_{o(y)'}) \mid |Sol([y],dp^j,G)|.$$  
     
 \end{proof}

 \begin{cor}
  \label{divv2222}

Let $G$ be a finite group of order $n$, and let $\pi$ be a subset of prime divisors of $n$. Suppose $G$ for any $p_i\in \pi$ possesses a Dedekind $A$-regular $p_i$-subgroup $N_i$ of order $p_i^{r_i}$ and exponent $p^{s_i}$. Let $\Pi_{p_i\in \pi}p_i^{s_i}=s$ and $\Pi_{p_i\in \pi}p_i^{r_i}=r$. Let $d$ be a divisor of $n_{s'}$, and $j$ be a divisor of $exp(G)_{s}$ such that $s \mid j$. Then $\frac{drj}{s} \mid |Sol(1,dj,G)|$.

 \end{cor}
 \begin{proof}
 
According to  Theorem \ref{divv22},
$\frac{j_pr_p}{s_p}\mid |Sol(1,j,G)|$ for any $p\in \pi$. It follows that 
$$\frac{drj}{s} =\Pi_{p\in \pi}\frac{j_pr_p}{s_p}\mid |Sol(1,dj,G)|.$$

 \end{proof}
 \begin{cor}
    Let $G$ be a finite group of order $n$, $P\in Syl_p(G)$ such that $exp(\Omega_1(P))=p$, and let $A$ be an elementary abelian $p$-subgroup of $G$ with maximal rank. Suppose there exists a divisor $d$ of $n_{p'}$ and integers $t$ and $j$ such that $p^j$ divides the exponent of $G$ and $|Sol(1,p^jd,G)|\leq p^{t}d$. Then, $|A|\leq p^{t-j+1}$.
 \end{cor}

 The following lemma demonstrates that if a Sylow $p$-subgroup of $G$ with odd order is not cyclic, then $G$ possesses an elementary abelian normal subgroup $N$ of order $p^2$.
 \begin{lem}\label{noncyc}
     Let   $P$ be a finite $p$-group of odd order.
     If $P$ is not  cyclic, then $P$  has a an elementary abelian    subgroup $N$ of order $p^2$.  

 \end{lem} 
 \begin{proof}
     It follows from  Theorem 4.10 of \cite{Gor}, that $P$ has a normal abelian non-cyclic subgroup $M$.
     Then $\Omega_1(M)$ is an elementary abelian normal subgroup of $P$ of order bigger than $p$. Therefore $\Omega_1(M)$ contains a normal subgroup $N$ of $G$ such that $N\cong C_p\times C_p$.
 \end{proof}
 To establish the second extension of the Frobenius' theorem, we rely on a crucial result by
Zemlin \cite{Zal}. Let $H$ be a Sylow $p$-group of $G$. We define $m(H,k,p)$: If 
$Sol(1,p^k,H)$ is a group, let its order  be $p^r$. Then $m(H,k,p)= 
min\{r,k(p-l)\}$. If $Sol(1,p^k,H)$  is not a group we let $m(H,k,p)= k(p-l)$.

\begin{lem}(see 3.8 in \cite{Zal})\label{Zalm}
     Let $G$ he a 
finite group, and $H$ a subgroup of $G$. Let $C$ be a class of elements 
of $G$ conjugate with respect to $H$. Let $p$ and $k$ be  positive integers where $p$ is a prime number.  Let $HyH$ be any double coset of $H$ in $G$. Then 
$|Sol(C,p^k,HyH)|$ is a multiple of 
$p^{m(H,k,p)}.$

\end{lem}
 \begin{thm}\label{Frob3}
    Let $G$ be a finite group of order $n$, and let $P\in Sylp(G)$. Let $d$ be a
divisor of $n_{p'}$
  and $p^j\mid |P|$. Then $|Sol(1,p^jd,G)|$ is a multiple of 
$dp^{m(G,k,p)}.$
 \end{thm}
 \begin{proof}
     We proceed by induction on $d$.
     First let $d=1$. 
     By Lemma \ref{Zalm}, $|Sol(1,p^j,GyG)|=|Sol(1,p^j,G)|$ is a multiple of 
$p^{m(G,j,p)}.$
So suppose that $d>1$.
Let $q$ be a prime divisor of $d$, and let 
$gcd(d_q,exp(G))=q^t$.
Then 
$$|Sol(1,dp^j,G)|=\sum_{i=0}^t|Sol(B_{q^i}(G),d_{q'}p^j,G)|.$$
By the induction hypothesis,
  $d p^{m(G,j)}\mid |Sol(1,d_{q'}p^j,G)|$. 
Let $1\leq i\leq t$, and let $y\in B_{q^i}(G)$.
By Proposition 2.3 of \cite{mohsen2}, $$Sol([y],d_{q'p},G)|=|y^G||Sol([y],d_yp^j,C_G(y))|$$ where $d_y=gcd(|C_G(y)|,d_{q'})$.
Let $H=C_G(y)$, $d_y=gcd(d_{o(y)'},|H|)$ and $gcd(p^j,|H|)=p^{j_y}$.
By the  induction hypothesis, 
 
$$p^{m(H,j,p)}\mid |Sol([y],d_yp^{j_y},H)|=|Sol(1,d_yp^{j_y},\frac{H}{\langle y\rangle})|.$$
 
Since $ |P|/|H|_p\mid |y^G|$, we deduce that 
$p^{m(G,j,p)}\mid |Sol([y],d_yp^{j_y},G)|$.
Consequently,
$p^{m(G,j,p)}\mid |Sol(B_{q^i}(G),d_{q'}p,G)|$. By Proposition 2.3 of \cite{mohsen2}, $d\mid |Sol(1,dp^j,G)|$,  so 
$dp^{m(G,j,p)}\mid |Sol(1,dp^j,G)|$.

 \end{proof}
 Unfortunately, in most cases where the prime number $p = 2$, Theorems \ref{divv22} and \ref{Frob3} do not provide more information than Frobenius' theorem regarding $|Sol(1, d2^j, G)|$. Therefore, we seek a different generalization of Frobenius' theorem for this specific case.
  \begin{lem}\label{lemmm}
     Let $G$ be a finite $2$-group such that $exp(G)<|G|/2$. Then for any positive integer $j$ such that $2^j\leq |G|/2$, we have $2^{j+1}$ dividing $|Sol(1,2^j,G)|$.
 \end{lem}
 \begin{proof}

    We proceed by induction on $|G|$.  If $|G|=8$, then $G\cong (C_2)^3$, and then every thing is trivial. So suppose that $|G|>8$.  Now, we proceed by  induction on $j$.
If $2^j=exp(G)$, then $2^{j+1}\mid |G|=|Sol(1,2^j,G)|$.
So suppose that the result is true for any divisor $2^r\leq |G|/2$ of $|G|$ such that $r>j$.
    There are comparatively few $2$-groups with no normal abelian subgroup of rank
$2$  such a $2$-group must be cyclic, dihedral, semidihedral, or generalized quaternion.
As,   $G$ does not contain any elements of order $|G|/2$, $G$ has an    elementary abelian normal subgroup $A$ of order $4$. 
    Let $x\in B_{2^{j+1}}(G)$.
First suppose that $G=A\langle x\rangle$.
Since $A\cong C_2\times C_2$, and $Aut(C_2\times C_2)=   PSL(2,2)\cong S_3$, we have 
$[x^2,A]=1$, so $A\langle x^2\rangle$ is an abelian group. Therefore $x^2\in Z(G)$. If $o(x)=2$, then $|G|=8$, , and thus $G\cong (C_2)^3$,  making the proof trivial. 
Therefore  $o(x)>2$.
If $o(x)<exp(G)$, then $G$ has an element of order $2\cdot o(x)=|G|/2$, which is a contradiction.
So $o(x)=exp(G)$.
Let $N=A\langle x^{o(x)/2}\rangle$.
If $2^{j}=o(x)$, then $2^{j+1}\mid |G|=|Sol(1,2^j,G)|$.
Thus   $2^j<o(x)$.
Clearly,   
$Sol(1,2^{j-1},\frac{G}{N})$ is a cyclic subgroup of $\frac{N\langle x\rangle}{N}$. Therefore
$Sol(1,2^{j-1},\frac{G}{N})=\langle aN\rangle$ where $a\in \langle x\rangle$.
If there exists $z\in B_{2^j}(G)\setminus N\langle x^2\rangle$ such that $\langle z\rangle \cap N=1$, then 
$\langle z\rangle \cap \langle x\rangle=1$. Hence, $\frac{G}{N}=\langle zN\rangle$.
Consequently, $o(zN)=o(xN)$.
If $j\leq 2$, then $o(x)\leq 8$, so $|G|\leq 32$. By a Gap computation, we have the result. So $j>2$.
Then $z^2\in N \langle x^2\rangle$. 
Since $\mho_1(N \langle x^2\rangle)\leq \langle x^2\rangle$, we have $\langle z\rangle \cap \langle x\rangle\neq 1$, which is a contradiction.
Hence, for any $z\in Sol(1,2^j,G)$, we have 
$o(zN)\mid 2^{j-1}$, so $zN\in Sol(1,2^{j-1},\frac{G}{N})$.
It follows that 
$2^{j+1}\mid |N\langle a\rangle|=|Sol(1,2^j,G)|.$
 
So $G\neq A\langle x\rangle$ for any $x\in B_{2^{j+1}}(G)$. 
Let $x\in B_{2^{j+1}}(G)$.
We claim that 
$2^{j+1}\mid |B_{2^{j+1}}(A[x^G])|$. Clearly, $\varphi(2^{j+1})=2^{j}\mid |B_{2^{j+1}}(A[x])|.$ There exist  $g_1,\ldots,g_k\in G$ such that $B_{2^{j+1}}(A[x^G])=A[x^{g_1}]\cup\ldots\cup A[x^{g_k}]$ and $A[x^{g_i}]\cap A[x^{g_j}]=\varnothing$ for all $1\leq i\neq j\leq k$.
If $2^{j+1}\mid |A[x]|$, then  
$$2^{j+1}\mid k|A[x]|=\sum_{i=1}^k|A[x^{g_i}]|=|B_{2^{j+1}}(A[x^G])|.$$

First suppose that 
$$exp(A\langle x\rangle)<\frac{|A\langle x\rangle|}{2}=\frac{|A|o(x)}{|A\cap \langle x\rangle| 2}=\frac{2\cdot o(x)}{|A\cap \langle x\rangle| }.$$
Then $A\cap \langle x\rangle=1$, and so 
  $exp(A\langle x^2\rangle)<|A\langle x^2\rangle|/2$.  
 By the induction hypothesis, $2^{j+2}\mid |Sol(1,2^{j+1},A\langle x\rangle)|$. As, $2^{j+1}\mid |A\langle x^2\rangle|$  by the Frobenius' theorem 
$2^{j+1}\mid |Sol(1,2^{j+1},A\langle x^2\rangle)|$. 

It follows that 
$$2^{j+1}\mid |Sol(1,2^{j+1},A\langle x\rangle)\setminus Sol(1,2^{j+1},A\langle x^2\rangle)|=|B_{2^{j+1}}(A[x])|.$$
  
Now, suppose that 
$exp(A\langle x\rangle)=|A\langle x\rangle|/2.$
If $[x,A]=1$, then $A\langle x\rangle$ is abelian, so
 $2^{j+1}\mid |A[x]|=|B_{2^{j+1}}(A[x])|.$
   
 So  $[x,A]\neq 1$.
Firstly,  suppose that $\langle x\rangle \lhd G$. Since $A\cong C_2\times C_2$ is normal in $A\langle x\rangle$, and $A\langle x\rangle$ is a $2$-group of maximal class,  either $A\langle x\rangle\cong M_s(2)=\langle a,b:a^{2^{s-1}}=b^2=1,a^b=a^{1+2^{s-2}}\rangle$ where $|A\langle x\rangle|=2^s>8$ or $A\langle x\rangle\cong D_8$. 
If $A\langle x\rangle\cong M_s(2)$, then  $\exp(\Omega_j(A\langle x\rangle))=2^j$, so
$2^{j+1}\mid |\Omega_j(A\langle x\rangle)|$, therefore 
$$2^{j+1}\mid |\Omega_{j+1}(A\langle x\rangle)\setminus \Omega_j(A\langle x\rangle)|=|B_{2^{j+1}}(A\langle x\rangle)|.$$

So $A\langle x\rangle\cong D_8$. Let $E=A\langle x\rangle.$
Since   $\langle x\rangle \lhd G$ and $x\not\in Z(E)$, we deduce that  $o(x)=4$.  Then $\langle x\rangle \cap E=\langle x^2\rangle$. 
Since $\frac{E}{A}\leq Z(\frac{G}{A})$ and 
$\frac{E}{\langle x\rangle}\leq Z(\frac{G}{\langle x\rangle})$, we deduce that $[G, E]\leq A\cap \langle x\rangle=\langle x^2\rangle=E'$.
By     Lemma 4.3 of \cite{Berke}, 
$G=E\ast C_G(E)$.
Let $v\in C_G(E)$ such that $v^2\in [x]$.
Then $(vx)^2=v^2x^2=x^2x^2=1$.
So $vx\in Sol(1,2,G)$.
Therefore  
\begin{eqnarray*}
|Sol(1,2,G)|&=&|Sol(1,2,E)||Sol(1,2,C_G(E))|/2+|B_4(E)||Sol([x^2],2,C_G(E)|/2\\&=&3|Sol(1,2,C_G(E))|+|Sol([x^2],2,C_G(E)|\\&=&
2|Sol(1,2,C_G(E))|+(|Sol(1,2,C_G(E))|+|Sol([x^2],2,C_G(E)|).
\end{eqnarray*}
As, $$|Sol(1,2,C_G(E))|+|Sol([x^2],2,C_G(E)|=2|Sol(1,2,\frac{C_G(E)}{N})|,$$ we conclude that 
\begin{eqnarray*}
|Sol(1,2,G)|&=&
2|Sol(1,2,C_G(E))|+2|Sol(1,2,\frac{C_G(E)}{N})|.   
\end{eqnarray*}
 Clearly, $4=2^{j+1}\mid |Sol(1,2,G)|.$
  
So $\langle x\rangle$ is not normal in $G$. Then 
 $2\mid [G:N_G(\langle x\rangle)]=k$. Therefore  
$2^{j+1}\mid |B_{2^{j+1}}(A[x^G])|$, as claimed.

We have 
$Sol(1,2^{j+1},G)\setminus Sol(1,2^j,G)=B_{2^{j+1}}(G)$. 
Since $B_{2^{j+1}}(G)$ is a union of disjoint subsets $B_{2^{j+1}}(A[u_1^G]),\ldots,B_{2^{j+1}}(A[u_r^G])$ for some  $u_1,\ldots,u_r \in B_{2^{j+1}}(G)$, we deduce that 
$2^{j+1}\mid |B_{2^{j+1}}(G)|.$
By the induction hypothesis, 
$2^{j+2}\mid |Sol(1,2^{j+1},G)|$, thus  
$$2^{j+1}\mid |Sol(1,2^j,G)|=|Sol(1,2^{j+1},G)|-|B_{2^{j+1}}(G)|.$$

 \end{proof}
 \begin{thm}\label{2va}
     Let $G$ be a finite group, and let $P$ be a Sylow $2$-subgroup of $G$ with exponent $2^t$. Consider $1\leq j\leq t$.
     If $P$ does not contain any  elements of order $|P|/2$, then $2^{j+1}d\mid |Sol(1,2^jd,G)|$ for any divisor $d$ of $n_{2'}$.
 \end{thm}
 \begin{proof}
We proceed by induction on $|G|$.
From Lemma \ref{lemmm}, we may assume that $|G|$ is not  a $2$-group. 
  Now, we proceed by induction on $2^j$.
If $2^j=|P|/2$, then by the Frobenius' theorem 
$$|P|\mid  |Sol(1,|P|,G)|=|Sol(1,|P|/2,G)|.$$
So suppose that the theorem is true for all divisors $2^r>2^j$ of $|P|/2$.
Now, we proceed by induction on $d$.
First suppose  $d=1$.
If $P$ is abelian, then by Theorem \ref{divv2}, $2^{j+1}\mid |Sol(1,2^{j+1},G)|$.
  Therefore $P$ is not abelian  and    $P$ is not of maximal class, as $exp(|P|)<|P|/2$.
  Then either $P$ has an elementary abelian subgroup  of rank three or by Theorem 1.1 of \cite{Janko}, $P$ has an abelian normal subgroup   of type $C_4\times C_2$ or $C_4\times C_4$. 
  We consider the following two cases:

  {\bf Case 1.}  The subgroup $P$ does not contain  any   elementary abelian subgroups  of rank three.
  Consequently, $P$ has a normal abelian subgroup $N$ of rank two and exponent four.

Let $U=B_{2^{j+1}}(G)$, and let $M$ be a   subgroup of $N$ of order $2^{r_M}$ and of  exponent $2^{s_M}$, and let $$V_M=\{b\in U:\langle b\rangle M\leq G, \ N_N(\langle b\rangle M)=M\}.$$
Let $b\in V_M$. 
We claim that $$2^{j+1}\mid [N:M]|B_{2^{j+1}}(M\langle b\rangle)|=|B_{2^{j+1}}((M\langle b\rangle))^N|.$$
Since $B_{2^{j+1}}(M\langle b\rangle)$ is a union of elements of order $2^{j+1}$,
 we have  $2^{j}\mid |B_{2^{j+1}}(M\langle b\rangle)|.$
Clearly,  $G\neq M\langle b\rangle$, as $M\langle b\rangle$ is a $2$-group.
First suppose that 
$exp(M\langle b\rangle)=|M\langle b\rangle|/2$. 
If $M\neq N$, then 
$2\mid [N:M]$, so
$2^{j+1}\mid [[N:M]|B_{2^{j+1}}(M\langle b\rangle)|.$
So  $N=M$.
If $N\langle b\rangle$ is abelian, then 
$2^{j+1}\mid |B_{2^{j+1}}(N\langle b\rangle)|$, as $N\langle b\rangle$ is not cyclic.
So $N\langle b\rangle$ is not abelian. Therefore 
$N\langle b\rangle$ is a $2$-group of maximal class. 
By Theorem 4.3(d) of   \cite{Gor},
  $N\langle b\rangle\cong M_s(2)=\langle h,z: h^{2^{s-1}}=z^2=1, h^z=h^{1+2^{s-1}}\rangle $ for $s\geq 4$.
Since  
$\Omega_1(N\langle b\rangle)\cong C_2\times C_2$ is a $A$-regular subgroup of $N\langle b\rangle$. From Theorem \ref{divv2}, 
$2^{j+1}\mid |Sol(1,2^j,N\langle b\rangle)|$, so $$2^{j+1}\mid  |Sol(1,2^{j+1},N\langle b\rangle)\setminus Sol(1,2^j,N\langle b\rangle)|=| B_{2^{j+1}}(N\langle b\rangle)|.$$

Hence, 
$2^{j+1}\mid [N:M]| B_{2^{j+1}}(M\langle b\rangle)|$.
So $exp(M\langle b\rangle)<|M\langle b\rangle|/2$.
By the induction hypothesis, $2^{j+1}\mid |Sol(1,2^j,M\langle b\rangle)|$, so $$2^{j+1}\mid  |Sol(1,2^{j+1},M\langle b\rangle)\setminus Sol(1,2^j,M\langle b\rangle)|=| B_{2^{j+1}}(M\langle b\rangle)|.$$

Hence, 
$2^{j+1}\mid [N:M]| B_{2^{j+1}}(M\langle b\rangle)|$, as claimed.

For any $b\in V_m$, since  $2^{j+1}\mid |B_{2^{j+1}}((M\langle b\rangle))^N|$ and $2^{j+1}\mid |B_{2^{j+1}}((M\langle b^2\rangle))^N|$, we deduce that 
$$2^{j+1}\mid |B_{2^{j+1}}((M\langle b\rangle))^N\setminus |B_{2^{j+1}}((M\langle b^2\rangle))^N|=|B_{2^{j+1}}((M[b]))^N|.$$

Clearly, $V_M$ is a union of    disjoint subsets $$B_{2^{j+1}}(M[ u_{1}^N]),\ldots,B_{2^{j+1}}(M[ u_{q}^N])$$
where $u_1,\ldots,u_q\in V_M$.
Hence $$2^{j+1}\mid  \sum_{i=1}^q|B_{2^{j+1}}(M[u_{i}^N])|=|V_M|.$$

Let $S$ denote the set of all subgroups of $N$. Consider $M,K \in S$ such that $M \neq K$, and let $a \in V_M \cap V_K$. Then, both $\langle a \rangle K$ and $\langle a \rangle M$ are subgroups of $G$. Since $N$ is an abelian group,   $H:=\langle a \rangle KM$ is a subgroup of $G$. Observing that $\langle a \rangle K$ is normal in $H$, we deduce $KM \leq N_N(\langle a \rangle K)=K$, which contradicts the assumption. Therefore, $V_M \cap V_K = \varnothing$.

Considering $B_{2^{j+1}}(G)=\bigcup_{M \in S}V_M$, and noting that $2^{j+1}$ divides $|V_M|$ for any $M \in S$, we infer that $2^{j+1}$ divides $|B_{2^{j+1}}(G)|$, as claimed.
It follows that 
$2^{j+1}\mid |Sol(1,2^j,G)|$.

{\bf Case 2.}  The group $G$   has  an    elementary abelian subgroup $N$  of rank three.
Let $U=B_{2^{j+1}}(G)$, and let $M$ be a   subgroup of $N$ of order $2^{r_M}$, and let $$V_M=\{b\in U:\langle b\rangle M\leq G, \ N_N(\langle b\rangle M)=M\}.$$
We claim that $$2^{j+1}\mid [N:M]|B_{2^{j+1}}(M\langle b\rangle)|=|B_{2^{j+1}}((M\langle b\rangle))^N|.$$

First suppose that 
$exp(M\langle b\rangle)=|M\langle b\rangle|/2$.
Then $[M:\varPhi(M)]\leq 4$, and hence $N\neq M$.
It follows from
$2\mid [N:M]$ that
$2^{j+1}\mid [[N:M]|B_{2^{j+1}}(M\langle b\rangle))|.$ 

So $exp(M\langle b\rangle)<|M\langle b\rangle|/2$.
By the induction hypothesis, $2^{j+1}\mid |Sol(1,2^j,M\langle b\rangle)|$, so $$2^{j+1}\mid  |Sol(1,2^{j+1},M\langle b\rangle)\setminus Sol(1,2^j,M\langle b\rangle)|=| B_{2^{j+1}}(M\langle b\rangle)|,$$
and  hence $$2^{j+1}\mid [N:M]|B_{2^{j+1}}(M\langle b\rangle)|=|B_{2^{j+1}}((M\langle b\rangle))^N|,$$
  as claimed.
 The rest of the proof of this case is similar to the first case.

So suppose that $d>1$.
Let $q$ be a prime divisor of $d$, and let $q^r=gcd(exp(G),d_q)$.
We have $Sol(1,2^jd,G)=\bigcup_{i=0}^rSol(B_{q^i}(G),2^jd_{q'},G)$. 
By then induction hypothesis, $2^{j+1}\mid |Sol(1,2^jd_{q'},G)|.$ Let $1\leq i\leq r$.
    Consider $y\in Sol(B_{q^i}(G),2^jd_{q'},G)$, and let $s=d_{(qo(y))'}$.
     Then $$|Sol([y],s,G)|=|y^G||Sol([y],2^js,C_G(y))|.$$

Let $R$ be a Sylow $2$-subgroup of the centralizer $C_G(y)$ with order $2^m$, and let $gcd(2^j,|C_G(y)|)=2^{j_y}$.
If the Sylow $2$-subgroups of $C_G(y)$ has an element of order $2^{j_y-1}$, then 
$2^{j-j_y+1}\mid |y^G|$, as Sylow $2$-subgroups of $G$ does not contain any elements of order $|P|/2$.
Hence 
$2^{j+1}\mid |Sol([y],s,G)|$.

If the Sylow $2$-subgroups of $C_G(y)$ does not contain any elements of order $2^{j_y-1}$, then by induction hypothesis,  
$2^{j_y+1}\mid |Sol(1,2^{j_y}s,\frac{C_G(y)}{\langle y\rangle})|=|Sol([y],2^{j_y}s,C_G(y))|$. Since $2^{j-j_y}\mid |y^G|$, we conclude that
$2^{j+1}\mid |Sol([y],s,G)|$.
According to Proposition 2.3 of  \cite{mohsen2}, $d$ divides $|Sol(1,d2^j,G)|$, so $2^{j+1}d$ divides $|Sol(1,d2^j,G)|$.

 \end{proof}
 \begin{thm}\label{22va}
     Let $G$ be a finite group. If $P\in Syl_2(G)$ does not contain any elements of order $|P|/2$, then for any divisor $d$ of $n_{2'}$,  we have $|Sol(1,2^jd,G)|\geq 2^{j+2}d$.

 \end{thm}
 \begin{proof}
 
Let $exp(P)=2^t$, and let $1\leq j\leq t$ be the smallest positive integer such that $ |Sol(1,2^jd,G)|< 2^{j+2}d$. Then we find that $ |Sol(1,2^{j-1}d,G)|\geq 2^{j+1}d$.
Since $2^j$ divides $exp(P)$, we have $B_{2^j}(G)\neq \varnothing$, so    \begin{eqnarray*}
    |Sol(1,2^{j}d,G)|&\geq& |Sol(1,2^{j-1}d,G)|+\varphi(2^j)\\&\geq& 2^{j+1}d+2^{j-1}\\&>& 2^{j+1}d.
\end{eqnarray*}  According to Theorem $\ref{2va}$, there exists a positive integer $k$ such that $ |Sol(1,2^jd,G)|= 2^{j+1}dk > 2^{j+1}d$, implying $k>1$. Therefore, $ |Sol(1,2^jd,G)|\geq 2^{j+2}d$, which contradicts our initial assumption.
     
 \end{proof}

To establish Lemma \ref{dis},  we require the following lemma.
 \begin{lem}(Fenchel  \cite{33}, and Murai \cite{66})\label{ciic}
     Let $G$ be a finite group and $e=|Sol(1,e,G)|$. Assume that $e$ divides $|G|$. If $p$ is a prime divisor of e and $\frac{|G|}{e}$, then Sylow $p$-subgroups of $G$ are cyclic, generalized quaternion, dihedral, or semi-dihedral.
 \end{lem}

First, we prove Question \ref{iss}, in  the case $p$ is the smallest prime divisor of $n$. 

\begin{lem}\label{dis}
Let $G$ be a finite group of order $n$, and let $P\in Syl_p(G)$ such that  $P$ is neither cyclic nor generalized quaternion where $p$ is the smallest prime divisor of $n$.  Let  $d$ be a positive divisor of $n$ which is co-prime to   $p$ and $p\leq p^m\leq exp(P)=p^t$.  If $p>2$, then  $|Sol(1,dp^{m},G)|\geq dp^{m+1}$ and if $p=2$, then $|Sol(1,n_{2'}2^{m},G)|\geq n_{2'}2^{m+1}$.

\end{lem}
\begin{proof}
Let $|P|=p^e$.
We will proceed with induction on  $n$. By Frobenius' theorem, $|Sol(1,dp^m,G)|=kdp^m$ for some positive integer $k$.
Firstly, let's suppose that $P$ possesses an elementary abelian   subgroup $N$ of order $p^2$.
By  Theorem  \ref{divv22}, if $p>2$, then 
$|Sol(1,dp^m,G)|\geq dp^{m+1}$. So let $p=2$ and let $d=n_{2'}$. 
If $|Sol(1,2^{m}d,G)| < 2^{m+1}d$, then  $k=1$, implying $|Sol(1,2^{m}d,G)| = 2^md$. Furthermore, by Lemma \ref{ciic} and  our assumption that $p=2$, we deduce that $P$ is either a dihedral $D_{2^e}$  or semi-dihedral group $SD_{2^e}$ with $e \geq 4$.  Therefore for $3\leq i\leq e-1$, $G$ has only one conjugacy classes of cyclic subgroups of order $2^i$.   
 We consider the following two cases:

{\bf Case 1.} First suppose that there exists  $1\leq i\leq m$  such that $|Sol(1,d2^i,G)|\geq 2^{i+1}d$.  We may assume that $i\leq m$ is the biggest integer such that $|Sol(1,d2^i,G)|\geq 2^{i+1}d$. If $i=m$, then we are done.
So suppose that $i<m$.
Then  $|Sol(1,d2^{i},G)|\geq s2^{i+1}$. Consider $a \in B_{2^m}(P)$. 
Since $o(a^{2^{m-i-1}})=2^{i+1}$,  by Proposition 2.3 in \cite{mohsen2}, 
$|Sol([a^{2^{m-i-1}}],d,G)|\geq \varphi(2^{i+1})d=2^id$. Thus
\begin{eqnarray*}
|Sol(1,d2^{i+1},G)| &\geq&  |Sol(1,d2^{i},G)|+|Sol([a^{2^{m-i-1}}],d,G)|\\&=&2^{i+1}d+2^id\\&>&2^{i+1}d.
\end{eqnarray*}
  
By the Frobenius theorem
$2^{i+1}d\mid |Sol(1,d2^{i+1},G)|$.
Therefore $|Sol(1,d2^{i+1},G)|=s2^{i+1}d$ where $s>1$.
It follows that $|Sol(1,d2^{i+1},G)|\geq 2^{i+2}d$,
which is a contradiction.

{\bf Case 2.}
Now, suppose for any  $1\leq i\leq m$  we have $|Sol(1,d2^i,G)|= 2^{i}d$. 
 If $G$ has two conjugacy classes of cyclic subgroups   $\langle g^G\rangle$ and $\langle v^G\rangle$ of order $2^i$ with    $1<2^i\leq 2^m$, then by  Proposition 2.3 in \cite{mohsen2}, $|Sol([w],d,G)|\geq 2^{i-1}d$ for $w\in \{g,v\}$. Consequently, $|Sol(1,2d,G)|>2^id$, so by the Frobenius theorem, $|Sol(1,2d,G)|\geq 2^{i+1}d$ which contradicts our assumption. Thus, we conclude that for any  $1\leq i\leq m$,   $G$ has only one conjugacy class of cyclic subgroups  $\langle y_i^G\rangle $ of  order $2^i$, where $y_i\in \langle x\rangle$ and   $x\in P$ has order $2^{e-1}$. Let $1\leq i\leq m$. 
If $|Sol([y_i],d,G)|>2^{i-1}d$, then $$|Sol(1,2^id,G)|=|Sol(1,2^{i-1}d,G)|+|Sol([y_i],d,G)|>2^{i}d,$$ which contradicts to our assumption.
So $|Sol(B_{2^i}(G),d,G)|=|Sol([y_i],d,G)|=2^{i-1}d$ for any $1\leq i\leq m$. 
If $m\geq 2$, then  
\begin{eqnarray*}
|G|&=&|Sol(1,2^{e-1}d,G)|\\&=&\sum_{i=0}^{e-1}|Sol([y_i],d,G)\\&=&|Sol(1,2d,G)|+\sum_{i=2}^{e-1}|Sol([y_i],d,G)|\\&=&
2d+ \sum_{i=2}^{e-1}d2^{i-1}\\&=&d+ \sum_{i=0}^{e-2}d2^{i}\\&=&
d+(2^{e-1}-1)d\\&<&
2^ed,
\end{eqnarray*}
which is a    contradiction. 

So $m=1$. If $B_4(G)=[y_2^G]$, then by a similar argument as the above we get a contradiction.
So $G$ has a two conjugacy classes of cyclyc subgroups  $\langle y_2^G\rangle$ and $\langle j^G\rangle$ of order $4$. 
Hence, for all $1\leq c\leq e-1$ with $c\neq 2$, $B_{2^c}(G)=[y_c^G]$.
According
to Proposition 2.3 in \cite{mohsen2}, $2d\mid gcd(|Sol([y_2],d,G)|,|Sol([j],d,G)|)$.
It follows that
$$|Sol(1,4d,G)|=|Sol(1,4d,G)|+|Sol([y_2],d,G)|+|Sol([j],d,G)|\geq 6d,$$ and so 
$|Sol(1,4d,G)|\geq 8d$, as $4d\mid |Sol(1,4d,G)|.$
Consequently,   
\begin{eqnarray*}
|G|&=&|Sol(1,2^{e-1}d,G)|\\&=&|Sol(1,4d,G)|+\sum_{i=3}^{e-1}|Sol([y_i],d,G)\\&=&
8d+ \sum_{i=3}^{e-1}d2^{i-1}\\&=&
5d+ \sum_{i=0}^{e-2}d2^{i-1}\\&=&
5d+(2^{e-1}-1)d\\&=&
(4+2^{e-1})d.
\end{eqnarray*}
Since $2^{e-1}\mid |G|$, we have $4+2^{e-1}=8$.
Then $e=3$. Since $G$ has two conjugacy classes of cyclic subgroups of order$4$,  $P=Q_8$, leading to the final contradiction.

 \end{proof}
 
 \begin{rem}
 Let $G$ be a finite group and let $d=p_1{^{\alpha_1}}\ldots p_k^{\alpha_k}$ be a divisor of $exp(G)$  where $p_1< \ldots<p_k$ are prime numbers and $\alpha_1,\ldots,\alpha_k$ are positive integers.  
We have $Sol(1,d,G)\setminus Sol(1,d/p_1,G)=Sol(B_{p_1^{\alpha_1}},d_{p_1'},G)$. Therefore 
$$Sol(1,d,G)=Sol(B_{p_1^{\alpha_1}},d_{p_1'},G)\cup Sol(1,d/p_1,G).$$
If $x\in  Sol(1,d/p_1,G)$, then $p_1^{\alpha_1}\nmid o(x)$ and for all $y\in Sol(B_{p_1^{\alpha_1}},d_{p_1'},G)$, we have $p_1^{\alpha_1}\mid o(y)$. It follows that $Sol(B_{p_1^{\alpha_1}},d_{p_1'},G)\cap Sol(1,d/p_1,G)=\varnothing$.
 
\end{rem}

We now study the case $p=2$ in Question \ref{iss}, when the Sylow $2$-subgroup  is a   generalized quaternion group.
    \begin{lem}\label{dec}
        Let $G$ be a finite group of order $n$, $d$ a divisor of $n_{2'}$ and    $Q_{2^e}\cong P\in Syl_2(G)$. If $|Sol(1,2d,G)|=2d$, then $G=K\rtimes P$ where $K$ is a   $2'$-subgroup of $G$.
    \end{lem}
    \begin{proof}

If $e=2$, then $P\cong C_2\times C_2$. By Theorem \ref{divv22},
$|Sol(1,2d,G)|\geq 4d$, which is a contradiction. So $e>2$.
Let $1\neq H\leq P$. If $H\ncong Q_8$, then $Aut(H)$ is a $2$-group, and so $\frac{N_G(H)}{C_G(H)}$ is a $2$-group.
Suppose  $H\cong Q_8$.
If $\frac{N_G(H)}{C_G(H)}$ is a $2$-group, then by the Frobenius normal $p$-complement theorem, $G=K\rtimes P$.   So $\frac{N_G(H)}{C_G(H)}$ is not a $2$-group.
Considering that $Aut(Q_8)=S_4$, there exists an element $y\in G$ of order $3$ such that $i^y = j$ and $j^y = ij$, where $H=\langle i,j\rangle$. We may assume that $i\in \langle a\rangle$ and $P=\langle a,j\rangle$ where $o(a)=2^{e-1}$.
Given that $\frac{P}{Z(P)}\cong D_{2^{e-1}}$ and $i^y=j$, all subgroups of $P$ of order $4$ are conjugate in $G$.
Let   $a\in P$ have an order of $2^{e-1}=|P|/4$. 
Consider $b\in \langle a\rangle$ with an order greater than $2$. Since $C_G(b)$ contains the unique element of order two, Sylow $2$-subgroups of $C_G(b)$ are cyclic. By the  Burnside normal $p$-complement theorem, $C_G(b)=K_b\rtimes S_b$, where $S_b\in Syl_2(C_G(b))$ and $K_b$ is a Hall $2'$-subgroup of $C_G(b)$. Let $d_b=gcd(|C_G(b)|,d)$. Then, 
$$|Sol([b],d,C_G(b))|=|b^G||Sol([b],d_b,C_G(b))|=\frac{n}{d_b2^{e-1}}d_bo(b)/2=do(b).$$

Hence, 
\begin{eqnarray*}
|Sol(1,d2^{e-1},G)|&=& |Sol(1,d2,G)|+ \sum_{i=2}^{e-1}|Sol([a^{2^{e-i-1}}],d,G)|\\&= &2d+\sum_{i=2}^{e-1}2^{i}d\\&=&2d+d(2^e-4)\\&=&d2^e-2d.
\end{eqnarray*} 
Since $d2^{e-1}\mid |Sol(1,d2^{e-1},G)|$, we have $d2^{e-1}\mid 2d$, so $e=2$, which is a contradiction.

    \end{proof}
      \section{Bijections}
  
 The following theorems generalize and confirm Question \ref{iss} for $p>2$. Let $d$ be a divisor of the integer number $n=\rho_1^{\theta_1}\ldots\rho_e^{\theta_e}$, where $\rho_1<\ldots<\rho_e$ are primes and $\theta_1,\ldots,\theta_e$ are positive integers.
To shorten the equations in the next theorem, we use the following notation:  For any divisor $d$ of $n$ and any $1\leq i\leq e$, the divisor  $\frac{d}{gcd(d,\rho_1^{\theta_1}\ldots\rho_i^{\theta_i})}=d_{(\rho_1\rho_2 \ldots\rho_i)'}$ of $d$ is denoted by $\lambda_i(d,n)$. Also, when there is no ambiguity,  $\lambda_i(d,n)$ simply  is denoted   by $\lambda_i(d)$.
\begin{thm}\label{inject}
Consider a finite group $G$ of order  $n=p_1^{\alpha_1} p_2^{\alpha_2} \ldots p_k^{\alpha_k}$ be a divisor of $n$  where $p_1 < p_2 < \ldots < p_k$ are prime numbers. Let $P_i \in Syl_{p_i}(G)$ and $N_i$ be an Dedekind $A$-regular   subgroup of $P_i$ of maximal rank of order  $p_i^{r_i}$ and of exponent $p_i$.
 
 \begin{enumerate}
     \item If $N_1$ is abelian, then 
  $$G\in \mathfrak{cl}(\Pi_{i=1}^k(C_{p_i^{\alpha_i-(r_i-1)}}\times (C_{p_i})^{r_i-1}).$$

    \item  If $N_1$ is not abelian,
   then 
  $$G\in \mathfrak{cl}(C_{2^{\alpha_1-(r_1-1)}}\times E\times Q_8\times \Pi_{i=2}^k(C_{p_i^{\alpha_i-(r_i-1)}}\times (C_{p_i})^{r_i-1})$$
  where $E=1$ whenever $r\leq 4$, and $E=(C_{2})^{r_1-4}$ whenever $r>4$.
\end{enumerate}
In addition,  if 
       $P_1$ is non-cyclic, and $|N_1|=2$, then 
  \begin{enumerate}
    \item  if $|Sol(1,2\cdot n_{2'},G)|>2\cdot n_{2'}$, then   $G\in \mathfrak{cl}(C_{2^{\alpha_1-1}}\times C_2\times  \Pi_{i=2}^k(C_{p_i^{\alpha_i-(r_i-1)}}\times (C_{p_i})^{r_i-1}))$.  

    \item  If $|Sol(1,2\cdot n_{2'},G)|=2\cdot n_{2'}$, then   $G\in \mathfrak{cl}( Q_{2^m}\times   \Pi_{i=2}^k(C_{p_i^{\alpha_i-(r_i-1)}}\times (C_{p_i})^{r_i-1}))$.

\end{enumerate}
\end{thm}
\begin{proof}
    
    Let   $X_{0,0}=\{1\}$. 
From Theorem  2.9 of  \cite{mohsen2},
 for all pairs $(i,j)$ with $1\leq i\leq k$ and $0\leq j\leq \alpha_i-1$, there exist a subset $X_{p_i,\alpha_i-j}$ of a   transversal for $C_{\lambda_i(n)}$ in $C_n$ and a bijection $f$ from $G$ onto $\bigcup_{i=1}^k\bigcup_{j=0}^{\delta_i-1}X_{\{p_i,\delta_i-j\}} C_{\lambda_i(n)}\cup X_{0,0}$ such that 
   
    $$o(z)\mid o(f(z)^{o(f(z)C_{\lambda_i(d) \cdot p_i^{\delta_i-j}})}),$$
for  all $z\in Sol(B_{p_i^{\delta_i-j}},\lambda_i(n),G)$.
Let $i\leq e\leq k$ and $0\leq j\leq \alpha_e-\theta_e$. 
First suppose that $2\nmid n$. From  Theorem \ref{divv2}, $ |Sol(1,p_e^j\lambda_e(n),G)|\geq \lambda_e(n)p_e^{r_e+j-1}$.
According to Lemma 2.7 of \cite{mohsen2}, we may assume that $C_{p_e^{r_e+j-1}}C_{\lambda_e(n)}\subseteq f^{-1}(Sol(1, p_e^j\lambda_e(n),G))$.

 Define  $\beta$ from $C_n$ onto $H:=\Pi_{i=1}^k(C_{p_i^{\alpha_i-(r_i-1)}}\times (C_{p_i})^{r_i-1})$ such that 
$$\beta(C_{p_i^{j+r_i-1}}C_{\lambda_i(n)})=(C_{p_i^{j}}\times (C_{p_i})^{r_i-1})W_i,$$
for all $i=1,\ldots,k$,  and $j=0,1,\ldots,\alpha_i-r_i$ where $W_i$ is a Hall subgroup of $H$ of order $\lambda_i(n)$.
Then 
 $o(z)\mid o(\beta(f(z))$
for  all $z\in Sol(B_{p_i^{\delta_i-j}},\lambda_i(n),G)$.

Now, suppose  $p_1=2$. Let $U:=C_{2^{j}}\times E\times Q_8$.
  Define  $\beta$ from $C_n$ onto $H:=(C_{2^{j+r_1-1}})=U\times \Pi_{i=2}^k(C_{p_i^{\alpha_i-(r_i-1)}}\times (C_{p_i})^{r_i-1})$ such that 
$$\beta(C_{p_i^{j+r_i-1}}C_{\lambda_i(n)})=C_{p_i^{j}}\times (C_{p_i})^{r_i-1}T_i,$$
and $\beta(T_i)=T_i$ 
for all $i=2,\ldots,k$,  $j=0,1,\ldots,\alpha_i-r_i$ where $T_i$ is a Hall subgroup of $H$ of order $\lambda_i(n)$ and $\beta(C_{2^{j+r_1-1}}C_{\lambda_1(n)})= U T_1$,
and $\beta(T_1)=T_1$      for all    $j=0,1,\ldots,\alpha_1-r_1$ where $T_1$ is a Hall subgroup of $H$ of order $\lambda_1(n)$.
    Then 
 $o(z)\mid o(\beta(f(z))$
for  all $z\in G$.

The proof of the second part follows a similar approach as outlined in the preceding cases, utilizing Lemma \ref{dis}.
\end{proof}
\begin{rem}\label{82}
In the previous theorem 
    if $|Sol(1,2\cdot n_{2'},G)|>2^m\cdot n_{2'}$, then  by the same argument we may show that $G\in \mathfrak{cl}(C_{2^{\alpha_1-1}}\times (C_2)^{m}\times  \Pi_{i=2}^k(C_{p_i^{\alpha_i-(r_i-1)}}\times (C_{p_i})^{r_i-1}))$.  
\end{rem}
\begin{thm}\label{inject221}

Consider a finite group $G$ of order $n$, and let $d=p_1^{\alpha_1} p_2^{\alpha_2} \ldots p_k^{\alpha_k}$ be a divisor of $n$, where $p_1 < p_2 < \ldots < p_k$ are prime numbers. Define $r_i=m(G,1,p_i)$ for $i=2,\ldots,k$, and let  $R:= \Pi_{i=1}^k(C_{p_i^{\alpha_i-(r_i-1)}}\times (C_{p_i})^{r_i-1})$. 
Additionally, if $P_1$ is a non-abelian $2$-group, let $d=n_{2'}$, and let  $r_1=3$ whenever $P_1$ has an abelian subgroup of rank $3$, otherwise, let $r_1=2$. 
Let  $R:=Q_{2^
{\alpha_1}}\times \Pi_{i=2}^k(C_{p_i^{\alpha_i-(r_i-1)}}\times (C_{p_i})^{r_i-1})$ whenever $|Sol(1,2d_{2'},G)|=2d_{2'}$, otherwise let $R:=C_{2^{\alpha_1-r_1+1}}\times (C_2)^{r_1-1}\times \Pi_{i=2}^k(C_{p_i^{\alpha_i-(r_i-1)}}\times (C_{p_i})^{r_i-1})$. 

Then there exists a subset $X$ of a transversal for $R$ in $C_{n/d}\times R$, along with a bijection $f$ from $Sol(1,d,G)$ onto $X R$, such that $o(x)\mid o(f(x))$ holds for all $x\in Sol(1,d,G)$.

\end{thm}
\begin{proof}
We will prove the case for $p_1=2$, as the argument for $p_1>2$ follows by similar reasoning as in the proof of Theorem \ref{inject}, utilizing Theorem \ref{Frob3}.

Let $X_{0,0}={1}$. First, suppose that $|Sol(1,2d_{2'},G)|>2d_{2'}$.   From Theorem 2.9 of \cite{mohsen2}, for all pairs $(i,j)$ with $1\leq i\leq k$ and $0\leq j\leq \alpha_i-1$, there exists a subset $X_{p_i,\alpha_i-j}$ of a transversal for $C_{\lambda_i(d)}$ in $C_n$ and a bijection $f$ from $Sol(1,d,G)$ onto $\bigcup_{i=1}^k\bigcup_{j=0}^{\delta_i-1}X_{\{p_i,\delta_i-j\}} C_{\lambda_i(d)}\cup X_{\{0,0\}}$ such that 
   
    $$o(z)\mid o(f(z)^{o(f(z)C_{\lambda_i(d) \cdot p_i^{\delta_i-j}})}),$$
for  all $z\in Sol(B_{p_i^{\delta_i-j}}(G),\lambda_i(d),G)$.
Let $i\leq e\leq k$ and $0\leq j\leq \alpha_e-\theta_e$. 
  From  Theorem \ref{divv2}, $ |Sol(1,p_e^j\lambda_e(d),G)|\geq \lambda_e(d)p_e^{r_e+j-1}$.
According to Lemma 2.7 of \cite{mohsen2}, we may assume that $C_{p_e^{r_e+j-1}}C_{\lambda_e(d)}\subseteq f^{-1}(Sol(1, p_e^j\lambda_e(d),G))$.

 Define  $\beta$ from $C_n$ onto $H:=C_{n/d}\times R$ such that 
$$\beta(C_{p_i^{j+r_i-1}}C_{\lambda_i(d)})=(C_{p_i^{j}}\times (C_{p_i})^{r_i-1})W_i,$$
for all $i=1,\ldots,k$,  and $j=0,1,\ldots,\alpha_i-r_i$ where $W_i$ is a Hall subgroup of $H$ of order $\lambda_i(d)$.
Then 
 $o(z)\mid o(\beta(f(z))$
for  all $z\in Sol(1,d,G)$.

Now, suppose  $|Sol(1,2d_{2'},G)|=2d_{2'}$. 
   We employ Lemma \ref{dec}, which states that if $|Sol(1,2d_{2'},G)|=2d_{2'}$, then $G=K\rtimes Q_{2^m}$ for some Hall subgroup $K$ of $G$. We proceed with a proof by induction on $d$, leveraging the preceding part.

\end{proof}

\begin{thm}\label{mmm}
Let $G$ be a  finite  group of order $n=p_1^{\alpha_1}\ldots p_k^{\alpha_k}$ where $p_1<\ldots<p_k$ are prime numbers, and let $P_i\in Syl_{p_i}(G)$ for all $i=1,\ldots,k$. Then $G\in \mathfrak{cl}(\Pi_{i=1}^k Q_i)$ where     for $i=1,2,\ldots,k$, we have
$Q_i=C_{p_i^{\alpha_1}}$ if    $P_i$ is cyclic and $Q_i=C_{p_i^{\alpha_i}}\times C_{p_i}$ whenever $P_i$ is not cyclic and either    $p_i>2$ or $p_i=2$ and $|Sol(1,n_{p_i'}p_i,G)|>2n_{2'}$,
if $Q_1$ is not cyclic, and  $|Sol(1,n_{p_i'}p_i,G)|=2n_{2'}$, then  $Q_1$ is  generaized quaternion group of order $2^{\alpha_1}.$ 
\end{thm}
\begin{proof}
    If $P_i$ is not cyclic or generalized quternion, then from Lemma \ref{noncyc}, $P$ has an elementary abelian  subgroup of order $p^2$.   Then the result follows from Theorem \ref{inject}.
\end{proof}

Now, we can confirm Question \ref{iss}.

 \begin{thm}
Let $G$  be a finite group of order $n$ and $p$ be a prime divisor of $n$. If a Sylow $p$-subgroup of $G$ is neither cyclic nor generalized quaternion, then   $G\in \mathfrak{cl}(C_{\frac{n}{p}} \times C_p)$. 
\end{thm}
\begin{proof}
If $p=2$, then the result follows from Theorem \ref{fmain}.
So $p>2$.
 Then $P$, has an elementary abelian normal subgroup of order $p^2$, and  from Theorem \ref{mmm}, 
   $G\in \mathfrak{cl}(\Pi_{i=1}^k Q_i)$ where     for $i=1,2,\ldots,k$, we have
$Q_i=C_{p_i^{\alpha_1}}$ if    $P_i$ is cyclic and $Q_i=C_{p_i^{\alpha_i}}\times C_{p_i}$ whenever $P_i$ is not cyclic. Since $\Pi_{i=1}^k Q_i$ is an abelian group, from main result of \cite{mohsen}, 
 $\Pi_{i=1}^k Q_i\in \mathfrak{cl}(C_{\frac{n}{p}} \times C_p)$. Consequently, 
 $G\in \mathfrak{cl}(C_{\frac{n}{p}} \times C_p)$. 
\end{proof}
\section{Applications}

Let $Gr(n)$ denote the set of all non-cyclic finite groups of order $n$.
For any positive integer $n$, define $\Psi_1(n):=\{G: \psi(H)\leq \psi(G), \ \forall H\in Gr(n)\}$.
For any positive integer $k\geq 2$, we define $\Psi_k(n)$ recursively as follows: $\Psi_k(n)=\{G: \psi(H)\leq \psi(G), \ \forall H\in Gr(n)\setminus \bigcup_{i=1}^{k-1}\Psi_i(n)\}$. 
\begin{lem}\label{same}
    Let $G\in \mathfrak{cl}(H)$ where $G$ and $H$ are two finite group.
    Then $\psi(G)=\psi(H)$ if and only if $G$ and $H$ are groups of the same order type.
\end{lem}
\begin{proof}
Let $f$ be a bijection from $G$ onto $H$ such that $o(x)\mid o(f(x))$. First suppose that $\psi(G)=\psi(H)$. If there exists $x\in G$ such that $o(x)<o(f(x))$, then 
$\psi(G)<\psi(H)$, which is a contradiction.

Clearly, if $G$ and $H$ are groups of the same order type, then $\psi(G)=\psi(H)$.
    
\end{proof}
\begin{lem}\label{pnil}
   Let $G\in \mathfrak{cl}(H)$ where $G$ and $H$ are two finite group such that  $\psi(G)=\psi(H)$.
   Let $p$ be a prime divisor of $|G|$. Then    $H$ is   a $p$-nilpotent group if and only if $G$ is a $p$-nilpotent group.  
\end{lem}
\begin{proof}
    From Lemma \ref{same}, $G$ and $H$ are groups of the same order type.
    From Corollary 3.4 of \cite{mohsen}, $H$ is   a $p$-nilpotent group if and only if $G$ is a $p$-nilpotent group.  
\end{proof}
The papers cited (\cite{1}, \cite{2}, \cite{10}, \cite{11}, and \cite{15}) explore the structural characteristics of elements of    $\Psi_2$. The subsequent remark illustrates how our findings can be effectively applied to determine  the structure of the elements of  $\Psi_2$
\begin{rem}

The principal focus of several articles is to characterize the structure of elements within $\Psi_2(n)$. Using our results, we provide a straightforward proof to classify all elements of $\Psi_2(n)$. Suppose $G\in \Psi_2(n)$. If $G$ possesses a non-cyclic Sylow $p$-subgroup $P$ for $p>2$ or $p=2$, and $P\ncong Q_{2^m}$ with $|P|=2^m$, then according to Theorem \ref{mmm}, $G\in \mathfrak{cl}(C_{n/p}\times C_p)$. Since $G\in \Psi_2(n)$, $\psi(G)=\psi(C_{n/p}\times C_p)$. By Lemma \ref{same} and Lemma  \ref{pnil}, $G$   is a nilpotent group. 
 Let $y\in C_{n/p}\times C_p$ of order $n/p$. There exists $u\in G$ such that $f(u)=y$. If $G$ does not contain   any elements of order $n/p$, then $o(u)<o(y)$, hence $G\not\in \Psi_2(n)$, leading to a contradiction. Thus, $o(u)=n/p$. Consequently, $G\cong C_{n/p}\rtimes C_p$. If $G$ is non-abelian, then $\psi(G)<\psi(C_{n/p}\times C_p)$, contradicting the assumption. Hence, $G\cong C_{n/p}\rtimes C_p$. It can then be easily shown that $p$ is the smallest prime divisor of $n$ such that $p^2\mid n$. Therefore, all Sylow subgroups of $G$ are cyclic or generalized quaternion.
     
 If $2\mid n$, and Sylow $2$-subgroup is generalized quaternion, then by Theorem \ref{fmain}, 
 $G\in \mathfrak{cl}(C_{n_{2'}}\times Q_{2^m})$.
 Since $G\in \Psi_2(n)$, $\psi(G)=\psi(C_{n_{2'}}\times Q_{2^m})$.
By Lemma \ref{same} and Lemma  \ref{pnil}, $G$   is a nilpotent group. 
Furthermore, all Sylow subgroups are the same order type.  
If $m>3$, then $\psi(C_{n_{2'}}\times Q_{2^m})<\psi(C_{2^{m-1}\cdot n_{2'}}\times C_2)$, which is a contradiction. Therefore  $G\cong C_{n_{2'}}\times Q_{8}$.
So suppose that all Sylow subgroups of $G$ are cyclic.
By Theorem \ref{fmain}, 
   $G\in \mathfrak{cl}(C_{n_{(pq)'}}\times (P\rtimes Q))$ where $P\in Syl_p(G)$ and $Q\in Syl_q(G)$. 
   By Lemma \ref{same} and Lemma  \ref{pnil}, 
    $G\cong C_{n_{(pq)'}}\times (P\rtimes Q).$

 \end{rem}
 Now we embark on the study of finite groups with the third-highest sum of element orders.
 We first establish that if $G\in \Psi_3(n)$ possesses either an abelian group of rank 3  two non-cyclic Sylow subgroups with distinct orders, then $G$ is classified as a nilpotent group.
 \begin{lem}\label{solli}
 Consider a finite group $G$ of order  $n=p_1^{\alpha_1} p_2^{\alpha_2} \ldots p_k^{\alpha_k}$ be a divisor of $n$  where $p_1:=2 < p_2 < \ldots < p_k$ and $\alpha_1>1$.
 If  $G$ has  an elementary  abelian subgroup of    order $p^3$
and   $\psi(G)\geq \psi(C_{n/p^2}\times  (C_p)^{2})$, then $G\cong C_{n/|P|}\times P$ is a nilpotent group.
   
 \end{lem}
 \begin{proof}
 Let $P\in Syl_p(G)$, and let $n_{p'}=d$.
   Let $U$ be an elementary abelian subgroup of $P$ of maximal rank. Then $|U|=p^j\geq p^3$.
   If $p=2$, then  
  from Theorem \ref{2va},  $|Sol(1,2^jd,G)|\geq 2^{j+2}d$, otherwise by Theorem \ref{Frob3}, $|Sol(1,p^jd,G)|\geq p^{j+2}d$, consequently, according to Theorem   
 \ref{inject221}, 
 $G\in \mathfrak{cl}(C_{n/p^2}\times  (C_p)^{2})$.
 Let $f$ be a bijection from $G$ onto $C_{n/p^2}\times  (C_p)^{2}$ such that $o(u)\mid o(f(u))$ for all $u\in G$.
 If   there exists $u\in G$ such that $o(u)<o(f(u))$, then 
 $\psi(G)<\psi(C_{n/p^2}\times  (C_p)^{2})$, which is a contradiction.
 So $o(x)=o(f(x))$ for any $x\in G$,   and so by Lemma \ref{same} and Lemma \ref{pnil}, $G\cong C_{n/|P|}\times P$ is a nilpotent group.

 \end{proof}
 
 \begin{lem}\label{s123}
    Let $G\in \Psi_3(n)$. If there exist   non-cyclic  subgroups  $P\in Syl_p(G)$ and $Q\in Syl_q(G)$ for $p<q$, then $G\cong C_{n/|Q||P|}\times P\times Q$ is a nilpotent group. 
\end{lem}
\begin{proof}
     Let $n=p_1^{\alpha_1}\ldots p_k^{\alpha_k}$ where $p_1<\ldots  <p_k$ are prime numbers.
If $|Sol(1,pn_{p'},G)|=2n_{p'}$, then  $p_1=p=2$.
By Lemmas   \ref{dis} and \ref{dec},
$G=K\rtimes Q_{2^{\alpha_1}}$.
Since $Q\in Syl_q(K)$ by Theorem \ref{mmm},  
  $G\in \mathfrak{cl}(C_{n/2^{\alpha_1}|Q|}\times Q_{2^{\alpha_1}} \times C_{|Q|/q}\times C_q)$.
It follows that $\psi(G)\geq \psi(C_{n/2^{\alpha_1}|Q|}\times Q_{2^{\alpha_1}} \times C_{|Q|/q}\times C_q)$, and so by Lemma \ref{same}, $G$ and $C_{n/2^{\alpha_1}|Q|}\times Q_{2^{\alpha_1}} \times C_{|Q|/q}\times C_q$ are the same order type, so $G\cong C_{n/2^{\alpha_1}|Q|}\times Q_{2^{\alpha_1}} \times C_{|Q|/q}\times C_q$ is a nilpotent group.

So $|Sol(1,pd,G)|\geq 4d$.
Then by Theorem \ref{mmm},
  $G\in \mathfrak{cl}(C_{n/|Q||P|}\times C_{|Q|/q}\times C_q\times C_{|P|/p}\times C_p)$.
Since $\psi(G)\geq \psi(C_{n/|Q||P|}\times C_{|Q|/q}\times C_q\times C_{|P|/p}\times C_p)$, from Lemma \ref{same}, $G$ and $C_{n/|Q||P|}\times C_{|Q|/q}\times C_q\times C_{|P|/p}\times C_p$ are the same order type, so $G$ is a nilpotent group, and hence  $G\cong C_{n/|Q||P|}\times P\times Q$. 
\end{proof}
In the following lemma  we show that all elements of $\Psi_3(n)$ are solvable groups. 
\begin{lem}\label{halp}
    Let $G\in \Psi_3(n)$. Then $G$ is a solvable group. 
\end{lem}
\begin{proof}
    We proceed by induction on $|G|$.
    By the Burnside $p$-complement theorem, we may assume that $4\mid n$.
    Let $n=p_1^{\alpha_1}\ldots p_k^{\alpha_k}$ where $p_1:=2<p_2<\ldots<p_k$ are prime numbers.
From   Lemma \ref{dec}, $|Sol(1,2n_{2'},G)|>2n_{2'}$.
So $|Sol(1,2n_{2'},G)|\geq 4n_{2'}$.
If $G$ has a non-cyclic Sylow $p$-subgroup $P$ for $p>2$, then by Theorem \ref{mmm},
  $G\in \mathfrak{cl}(C_{n/2|P|}\times C_2\times C_{|P|/p}\times C_p)$.
Since $\psi(G)\geq \psi(C_{n/2|P|}\times C_2\times C_{|P|/p}\times C_p)$, from Lemma \ref{same}, $G$ and $C_{n/2|P|}\times C_2\times C_{|P|/p}\times C_p$ are the same order type, so $G$ is a nilpotent group.
So suppose that all Sylow $p$-subgroups   for $p>2$ are cyclic. 
Let $P\in Syl_2(G)$.
If $P$ has an elementary abelian $2$-group of rank 3, then from  Lemma \ref{solli}, $G$ is a nilpotent group.
So suppose that $P$ does not contain  elementary abelian subgroups of rank 3. Referring to Theorem 4.10 in \cite{Gor},  we deduce that $P$ is either a dihedral or semi-dihedral group with $\alpha_1 \geq 4$ or generalized quaternion. 
By Theorem \ref{inject221},
$G\in \mathfrak{cl}(C_{n/2}\times C_2)$. If $3\nmid [N_G(H):C_G(H)]$ for all subgroups $H$ of $P$, then by the Frobenius normal $p$-complement theorem,   $G$ is a $2$-nilpotent group, so $G$ is a solvable group.
So suppose that $3\mid [N_G(H):C_G(H)]$ for some subgroup $H$ of $P$, then $p_2=3$.
If $\psi(G)=\psi(C_{n/2}\times C_2)$, then $G$ and $C_{n/2}\times C_2$ are the same order type and so $G$ is a solvable group.
So $\psi(G)<\psi(C_{n/2}\times C_2)$, consequently, 
$C_{n/2}\times C_2\in \Psi_2(n)$, as $G\in \Psi_3(n)$.
Form Theorem 1.2 of \cite{mohsen234},
if $3^{\alpha_2}>3$ or $2^{\alpha_1}>8$, then $C_{n/2}\times C_2\not\in \Psi_2(n)$, which is a contradiction.
So 
$3^{\alpha_2}=3$ and  $2^{\alpha_1}=8$.
It follows that $P_1\cong D_8$, as $H$ has a subgroup of type $C_2\times C_2$.
From Theorem 1 of \cite{Gorr}, 
 $\frac{G}{O_2(G)}$ is either a $2$-group, or has a normal subgroup
$N$ of odd index isomorphic to $PGL_2(q)$ or
$PSL_2(q)$ for some power $q$ of an odd prime, as Sylow $3$-subgroup of $A_7$ is not cyclic.
If $O_2(G)\neq 1$, then $|\frac{P_1}{O_2(G)}|\leq 2$, as $3\mid [N_G(H):C_G(H)]$, so by the Buernside $p$-complement theorem $G$ is a solvable group.
So $O_2(G)=1$. 
Clearly, we may assume that $G$ is not a $2$-group.
Let $m:=(q+1)\cdot gcd(|Z(N)|,2)\cdot [G:N]$.
Then for any $x\in G$, we have 
$o(x)\leq m$.
If $q\leq 3$, then $G$ is a solvable group. So $q-1\geq 3$, and then $m\leq n/3$.
Hence, 
$$\psi(G)<mn\leq  3\varphi(n/3)(n/3)\leq |B_{n/3}(C_{n/3}\times C_3)| \leq \psi(C_{n/3}\times C_3).$$
It follows that $C_{n/3}\times C_3\in \Psi_2(G)$,  which is a contradiction.

\end{proof}
The following two lemmas aid  in proving Theorem \ref{fmain}.
 
  \begin{lem}\label{allcyc}
Let $G$ be a finite non-cyclic group of order n such that all Sylow subgroups
of $G$ are cyclic. Let $P=\langle u\rangle \in  Sylp(F it(G))$ and $Q=\langle v\rangle\in  Sylq(G)$ such that $QP$ is not cyclic. Then
$G\in \mathfrak{cl}(C_{\frac{n}{|QP|}} \times QP)$.

       \end{lem}
   \begin{proof}
There exist elements $a$ and $b$ in $G$ such that $G = (\langle a\rangle \times P) \rtimes  (\langle b\rangle \times  Q)$ where
$Fit(G)=\langle a\rangle \times P$. We can assume that $C_{\frac{n}{|QP|}}= \langle a\rangle \times   \langle b\rangle$. Define $f$ from $G$ onto $C_{\frac{n}{|QP|}}\times QP$ by 
$f(a^iu^jb^rv^s)=a^ib^ru^jv^s$ for all integers $i,j,r$ and $s$. By a straightforward induction on $|G|$, we can establish that $o(a^iu^jb^rv^s)\mid o(f(a^iu^jb^rv^s)).$

   \end{proof}

\begin{lem}\label{2cm}
    Let $G$ be a finite solvable group of order $n$. If 
    $exp(P)=|P|/2$ and $3\nmid n$, then $G$ is a $2$-nilpotent group. Also, if $3\mid n$, then 
    $G=T\rtimes (QP)$ where $Q\in Syl_3(G)$ and $T$ is a Hall complement for $QP\leq G$.
    
\end{lem}
\begin{proof}
    We proceed by induction on $n$.  
    First suppose that $Fit(G)\nleq P$.
    Let $N$ be a normal minimal subgroup of $G$ such that $2\nmid |N|$, and let $K$ be a $2'$-Hall subgroup of $G$. First suppose that $3\nmid n$. By the induction hypothesis,  $\frac{G}{N}=\frac{K}{N}\rtimes \frac{PN}{N}$, and so $G=K\rtimes N$. 
    If $N\leq K$  then $G=K\rtimes P$.  

 So suppose that $Fit(G)\leq P$. Let $N\leq P$ be a normal minimal subgroup of $G$.
First suppose that $|N|=2$.
 By the induction hypothesis,  if $3\nmid n$, then  $\frac{G}{N}=\frac{KN}{N}\rtimes \frac{P}{N}$.
By the   Burnside $p$-complement $KN=K\rtimes N$, so $K\lhd G$.
    So $N\cong C_2\times C_2$.
Since $N$ is the unique normal minimal subgroup of $G$ such that $N\nleq \varPhi(G)$, we have $G=N\rtimes M$ where $M$ is a maximal subgroup of $G$.
It follows that $P=N=Fit(G)$.
We have  $Aut(N)=PSL(2,2)$.
Since  $3\nmid n$, then  $[N_G(H):C_G(H)]$ is a $2$-group for all subgroups $H$ of $N$. By the  Frobenius $p$-complement theorem $K\lhd G$.

Now, suppose $3 \mid n$. By a similar argument as in the first case, we assume that $N = Fit(G)$. Since $Aut(N) = PSL(2,2) \cong S_3$, it follows that $[N_G(H) : C_G(H)] \mid 6$. Let $U$ be a Hall $\{2,3\}$-subgroup of $C_G(N)$. Given that all Sylow subgroups of $\frac{C_G(N)}{N}$ are cyclic, Burnside’s $p$-complement theorem implies that $C_G(N) = T \rtimes U$.
Since $C_G(N) \lhd G$, we conclude that $G = T \rtimes (QP)$.

\end{proof}

\begin{thm}\label{pp2}
  Let $G$ belong to $\Psi_3(n)$. Assume $p$ denotes the smallest prime divisor of $G$ such that its Sylow $p$-subgroup $P$ is non-cyclic. Consequently, $G$ decomposes as $G=K\rtimes P$, where $K$ is a Hall $p'$-subgroup of $G$. Additionally, if all Sylow subgroups of $K$ are cyclic, then $K$ is cyclic.
\end{thm}
\begin{proof}
We proceed by induction on $n$. By Lemma \ref{s123}, we can assume that for any prime divisor $q \neq p$ of $n$, the Sylow $q$-subgroups are cyclic. Furthermore, by Lemma \ref{halp}, $G$ is a solvable group.
By Lemma \ref{solli}, we may assume that $P$ does not contain any   abelian group of rank $3$.
First   suppose  that   $G=K\rtimes P$. If    $K=U\rtimes V$ is not a cyclic group, then   $\psi(G)\leq \psi(K)\psi(P)<\psi(C_{|P|}\times K)$ and  $K\not\in \Psi_2(|K|)$,   so $C_{|P|}\times K\in \Psi_3(n)$, which is a contradiction.
    So $K$ is a cyclic group.  So suppose that $G\neq K\rtimes P$.
We consider the following two cases:

{\bf Case 1.}
First suppose that $p=2$. If $|Sol(1,2n_{2'},G)|> 4n_{2'}$, then by Remark \ref{82}, $G\in \mathfrak{cl}(C_{n/8}\times (C_2)^8)$, so $\psi(G)=\psi(C_{n/8}\times (C_2)^8)$, and so $G$ is a nilpotent group.
So $|Sol(1,2n_{2'},G)|\leq 4n_{2'}$. By Lemmas   \ref{dis} and \ref{dec}, and our assumption 
 $|Sol(1,2n_{2'},G)|> 2n_{2'}$. So $|Sol(1,2n_{2'},G)|= 4n_{2'}$.
Then by Theorem \ref{22va}, $P$ is generalized quaternion, dihedral, or semi-dihedral. 
By Lemma \ref{2cm}, we may assume that $3\mid n$ and $G=T\rtimes (QP)$ where    or  $Q=\langle u\rangle\in Syl_3(G)$ of order $3^{m_2}$. Furthermore, if $Q\neq 1$, then  $QP$ is not a nilpotent group.
If $[T,QP]\neq 1$, then by Lemma \ref{allcyc}, $\psi(G)<  \psi(T\times (Q P))$, which is a contradiction.
So $[T,QP]=1$.
Hence, we may assume that $G=QP$.
First suppose that $m_2>1$.
Let $\sigma\in Aut(Q)$ such that $(u^i)^{\sigma}=u^{i+i3^{m_2-1}}$ for all integers $i$.
Then $o([u,\sigma])=3$.
Let $H=C_{\frac{n}{3^m\cdot 2}}\times (Q\rtimes_{\sigma}\langle \sigma\rangle)$.
As, $\psi(G)<\psi(H)$, we have $H\in \Psi_2(G)$.
Since $H$ is not a nilpotent group, by Theorem 1.1 of \cite{15}, we have $m_2 = 1$, which leads to a contradiction. Therefore, $m_2 = 1$. If $m_1 \leq 4$, then $|G|\leq 48$, the result follows from a GAP computation.
 So $m_1\geq 5$.
According to  Theorem  1.1 of \cite{15}, 
$C_{n/2}\times C_2\in \Psi_2(n)$.
Let $Y=C_{|P|/2}\times S_3$. 
As, $3\mid [N_G(N):C_G(N)]$, we have 
  $exp(G)\leq n/6$. Then 
$$\psi(G)<n^2/6\leq \varphi(n/2)n/2+\varphi(n/3)2=|B_{n/2}(Y)|+|B_{n/3}(Y)|<\psi(Y),$$
which is a contradiction.

{\bf Case 2.}
So   $p>2$. 
If there exists a prime divisor $r<p$, then by the Burnside $p$-complement theorem, $G=K_1\rtimes R$ where $R\in Syl_r(G)$.
Clearly, $\psi(G)\leq \psi(K_1)\psi(R)=\psi(K_1\times R)$.
If $K_1$ is a nilpotent group, then we are done.
So $K_1$ is not a nilpotent group. Therefore $\psi(G)<\psi(K_1\times R)$.
Furthermore,  as $P$ is not cyclic, and $K_1$ is not nilpotent,  $K_1\times R\in \Psi_3(n)$, which leads  a contradiction.
So $p$ is the smallest prime divisor of $n$. By the Frobenius $p$-complement theorem, $P$ has a subgroup $H$ such that $\frac{N_G(H)}{C_G(H)}$ is not a $p$-group. 
Let $q|\mid \frac{N_G(H)}{C_G(H)}|$ be a prime number, and let $Q\in Syl_q(G)$.
Using similar arguments to those in Lemma \ref{2cm}, we can show that $G = T \rtimes (QP)$ where $[Q, P] \neq 1$. The remainder of the proof follows the same reasoning as in case 1.
    
\end{proof}
Now, the proof of Theorem \ref{main5}
follows from Lemmas \ref{s123}, \ref{s123} and \ref{halp} and Theorem \ref{pp2}.
Let $f$ be a function from $\mathbb{R}$ to $\mathbb{R}$, and let $P_k(G)$ be the set of all subsets of $G$ of size $k$. Recall that $\psi_{f,k}(G) = \sum_{S \in P_k(G)} \mu(S)$, where $G$ is a finite group and $\mu(S) = \prod_{x \in S} f(o(x))$.    
We need the following two Lemmas to prove Theorems \ref{sec} and \ref{oo}:
\begin{lem}\label{psi}(lemma 2.5, of \cite{mohs})\label{dir}
Let $G=A\times B$. Then $\psi^{I,l}(G)\leq \psi^{I,l}(A)\psi^{I,l}(B)$
 and $\psi^{I,l}(G)=\psi^{I,l}(A)\psi^{I,l}(B)$ if and only if $gcd(|A|,|B|)=1$.
\end{lem}
 
 The preceding lemma does not hold for $\psi_{I,l}$ when $k>1$. For instance, consider $\psi_{I,2}(C_6)=173$, $\psi_{I,2}(C_3)=15$, and $\psi_{I,2}(C_2)=2$. Consequently, extracting information about $G$ using $\psi_{I,l}(G)$ becomes significantly more intricate.

Hereafter, let $I$ denote the identity map from $\mathbb{Z}$ to $\mathbb{Z}$.
 \begin{lem}\label{reduce}
     Let $G$ be a finite $p$-group such that 
     $\Omega_1(G)$ is a subgroup of $G$ of order $p^r$ and  exponent $p$.
     Then $$\psi^{I,l}(G)=1-p^l+p^{r+l}\psi^{I,l}(\frac{P}{N}).$$
 \end{lem}
 \begin{proof}
     
Let $N:=\Omega_{1}(G)$. Then we have $<x>\cap N\neq 1$ for all $1\neq x \in P$. Since   $<x^{\frac{o(x)}{p}}>$ is a subgroup of $<x>\cap N$ . Let  $X$ be a left  transversal for $N$ in  $P$ such that $1\in X$. Let  $1\neq x\in X$. Then  $o(x)\geq p^{2}$ since $N$ does not contain $x$. If  $y\in N$, then  
 $o(xy)=o(x)$. This implies that  $$\psi^{I,l}(G)=\sum_{x\in X}\psi^{I,l}(xN)=\psi^{I,l}(N)+|N|\sum_{1\neq x\in X}o(x).$$ If $1\neq x\in X$, then $<x>\cap N\neq 1$ which follows that $o(x)^l=p^lo(xN)^l$. We have $$\psi^{I,l}(N)=1+(p^r-1)p^l=1-p^l+p^{r+l}.$$ Hence

\begin{eqnarray*}
\psi^{I,l}(G)&=&\psi^{I,l}(N)+|N|\sum_{1\neq x\in X}o(x)\\
&=&\psi^{I,l}(N)+|N|p^l\sum_{1\neq x\in X}o(x N)\\
&=&\psi^{I,l}(N)+p^{r+l}(\psi^{I,l}(\frac{G}{N})-1)\\&=&
1-p^l+p^{r+l}\psi^{I,l}(\frac{G}{N}).
\end{eqnarray*}
 \end{proof}
 \begin{cor}\label{siss}
     Let $G=C_{p^{\alpha-r+1}}\times (C_{p})^{r-1}$ be a $p$-group.
     Then $$\psi^{I,l}(G)=1-p^l+p^{r+l}[\frac{p-1}{p}\frac{p^{(l+1)(\alpha+1)}-1}{p^{l+1}-1}].$$
 \end{cor}
 \begin{proof}
Let $N=\Omega_1(G)$. We have 
\begin{eqnarray*}
\psi^{I,l}(C_{p^{\alpha}})&=&\varphi(p^{\alpha})p^{l\alpha}+\ldots+\varphi(p)p^{l}+1\\&=&\frac{p-1}{p}(p^{(l+1)\alpha}+p^{(l+1)(\alpha-1)}+\ldots+p^{l+1}+1\\&=&\frac{p-1}{p}\frac{p^{(l+1)(\alpha+1)}-1}{p^{l+1}-1}.
\end{eqnarray*}

From   Lemma \ref{reduce},
$$\psi^{I,l}(G)=1-p^l+p^{r+l}[\frac{p-1}{p}\frac{p^{(l+1)(\alpha+1)}-1}{p^{l+1}-1}].$$. 
 \end{proof}

\begin{thm}\label{sec}
Let $G$ be a  finite  group of order $n=p_1^{\alpha_1}p_2^{\alpha_2}\ldots p_k^{\alpha_k}$ where $p_1<\ldots <p_k$ are prime numbers. 
Let $N_i$ be an elementary abelian    $p_i$-subgroup of $G$ of  order $p_i^{r_i}\leq p_i^{p_i-1}$  for all $i = 1, 2, \ldots, k$. 
Let $Q_i=C_{p_i^{\alpha_i-r_i+1}}\times (C_{p_i})^{r_i-1}$ for all $i=1,\ldots,k$.
If $l\in \mathbb{N}$, then 

\begin{enumerate}
    \item 
 If $f$   is an increasing  function, then   $$\psi_{f,l}(G)\leq  \psi_{f,l}(\Pi_{i=1}^k (Q_i)\ \textit{and}\    \psi^{f,l}(G)\leq \Pi_{i=1}^k \psi^{f,l}(Q_i).$$ 

  \item  If $f$   is a  decreasing  function, then   $$\psi_{f,l}(G)\geq \psi_{f,l}(\Pi_{i=1}^k (Q_i)\ \textit{and}\    
  \psi^{f,l}(G)\geq \Pi_{i=1}^k \psi^{f,l}(Q_i).$$ 
\end{enumerate}
 \end{thm}
\begin{proof}

 Let $H=\Pi_{i=1}^k(C_{p_i^{\alpha_i-r_i+1}}\times (C_{p_i})^{r_i-1}$.
 By Lemma 9.4 of \cite{Berke},
 for any $y\in Sol(1,|P_i|,G)$, we have 
 $exp(\Omega_1(\langle y\rangle N_{N_i}(\langle y\rangle)))=p_i$.
   From Theorem  \ref{mmm}, there exists a bijection $\beta$   from $G$ onto
   $H$ such that $o(x)\mid o(\beta(x))$ for all $x\in G$.
   Then
    \begin{eqnarray*}
\psi^{f,l}(G)&=&\sum_{S \in P_l(G)} \mu(S)\\&=&
\sum_{S \in P_l(G)} \Pi_{x\in S}f(o(x))\\&\leq&
\sum_{S \in P_l(G)} \Pi_{x\in S}\beta(f(o(x)))\\&=&
\psi^{f,l}(H).
\end{eqnarray*}
By Lemma \ref{dir}, $$\psi^{f,l}(H)=\Pi_{i=1}^k\psi^{f,l}((C_{p_i^{\alpha_i-r_i+1}}\times (C_{p_i})^{r_i-1}).$$
    
  \end{proof}
  
We are now in a position to prove Theorems \ref{sec} and \ref{oo}, along with Corollary \ref{coo}. 
\begin{proof}
   
When $f=I$ represents the identity map and $k=1$, applying Lemma \ref{dir} and Corollary \ref{siss}, we deduce that:
\begin{eqnarray*}
 \psi^{I,1}(G)=\psi(G)&\leq&  \Pi_{i=1}^k\psi((C_{p_i^{\alpha_i-r_i+1}}\times (C_{p_i})^{r_i-1}))\\&=&\Pi_{i=1}^k(1-p_i+p_i^{r_i}[ \frac{(p_i-1)(p_i^{2(\alpha_i+1)}-1)}{p_i^{2}-1}])\\&=&
 \Pi_{i=1}^k(1-p_i+  \frac{p_i^{r_i}(p_i^{2(\alpha_i+1)}-1)}{p_i+1})
 .
\end{eqnarray*}
    Consequently,  $$o(G)\leq \Pi_{i=1}^k(\frac{1-p_i}{p_i^{\alpha_i}}+p_i^{r_i-\alpha_i}( \frac{ (p_i^{2(\alpha_i+1)}-1)}{p_i+1})).$$

\end{proof}

\bibliographystyle{amsplain}

\begin{thebibliography}{10}
 
\bibitem{jaf}
H. Amiri, S.M. Jafarian Amiri, I.M. Isaacs, Sums of element orders in finite groups, Comm. Algebra 37 (9)  2978-2980(2009).
\bibitem{mohsen234}
S.M. Jafarian Amiri, M. Amiri, Second maximum sum of element orders on finite groups, J. Pure
Appl. Algebra 218 (3) (2014), 531-539.
\bibitem{mohsen2}
M. Amiri,   On a  bijection between   a finite group    and   cyclic group, Journal of Pure and Applied Algebra,Volume 228, Issue 7, July 2024, 107632.


\bibitem{man2}M. Amiri, I. Kushuba, I. Lima, On $LC$-Subgroup of a Periodic group, 
https://doi.org/10.48550/arXiv.2212.03104.


  \bibitem{6}  M. Baniasad Asad, B. Khosravi, A Criterion for Solvability of a Finite Group by the Sum of Element
Orders, J. Algebra 516 (2018), 115-124.

\bibitem{Janko} Z. Janko, Finite 2-Groups with No Normal Elementary Abelian
Subgroups of Order 8, Journal of Algebra 246, 951–961 (2001).


\bibitem{Berke}Y. Berkovich and  Z. Janko, Groups of Prime Power Order Volume 1.

 \bibitem{33} K. Fenchel, On a theorem of Frobenius, Math, Scand, 42 (1978), 243-250.
\bibitem{fr}
F. G. Frobenius, Verallgemeinerung des Sylowschen Satzes, Berliner Sitz. 981-993  (1895).

 
\bibitem{18} G. Chen, On Thompson’s Conjecture -For sporadic groups, in: Proc. China Assoc. Sci. and Tech.
First Academic Annual Meeting of Youths, pp. 1–6, Chinese Sci. and Tech. Press, Beijing  (1992).
\bibitem{7} C.Y. Chew, A.Y.M. Chin, C.S. Lim
A recursive formula for the sum of element orders of finite Abelian groups
Results Math., 72 (2017), pp. 1897-1905.
\bibitem{Gor} D. Gorenstein, finite groups, American Mathematical Soc., 2007, Volume 301. 


\bibitem{1}S.M. Jafarian Amiri
Second maximum sum of element orders on finite nilpotent groups
Commun. Algebra, 41 (2013), pp. 2055-2059.
\bibitem{mohs}S.M. Jafarian Amiri, M. Amiri, Sum of the products of the orders of two distinct elements in finite
groups, Comm. Algebra 42 (12) (2014), 5319-5328.
\bibitem{mohsen}  S. M. Jafarian Amiri and M.Amiri, Characterization of finite groups by a bijection
with a divisible property on the element orders, Comm. Algebra 45(8)  3396-3401(2017).

 \bibitem{2} S.M. Jafarian Amiri, M. Amiri, Second maximum sum of element orders on finite groups, J. Pure Appl. Algebra 218 (2014)
531–539.
\bibitem{Gorr} D. Gorenstein and J. Walter, The characterization of finite groups with dihedral
Sylow 2-subgroups I., II., III., J. Algebra 2 (1965), 85-15 1, 218-270, 334-393.
 
 \bibitem{9} M. Herzog, P. Longobardi, M. Maj, Two new criteria for solvability of finite groups in finite groups,
J. Algebra 511 (2018), 215-226.

 \bibitem{10} M. Herzog, P. Longobardi, M. Maj
Sums of element orders in groups of order 2m with m odd
Commun. Algebra, 47 (2019), pp. 2035-2948.


\bibitem{11}M. Herzog, P. Longobardi, M. Maj
The second maximal groups with respect to the sum of element orders
J. Pure Appl. Algebra, 225 (2021), Article 106531
 
\bibitem{I} I. Martin Isaacs,   Finite Group Theory,
American Mathematical Soc., Vol. 92  (2008).
 \bibitem{66} M. Murai, On a conjecture of Frobenius (in Japanese), S$\hat{u}$gaku 35 (1983), 82-84. 
\bibitem{Kishore}H. Kishore Dey and A. Mondal,  An exact upper bound for the sum of powers of element orders in
non-cyclic finite groups, Journal of Pure and Applied Algebra
Volume 228, Issue 6, June 2024, 107580.

   
\bibitem{lad} F. Ladisch, Order-increasing bijection from arbitrary groups to cyclic groups, http://mathoverflow.net/a/107395.
\bibitem{Marefat2}Y. Marefat, A. Iranmanesh, A. Tehranian, On the sum of element orders of finite simple groups, J.
Algebra Appl. 12 (7) (2013), 135-138.
 
   
\bibitem{Kh} V. D. Mazurov, E. I. Khukhro,   The Kourovka Notebook. Unsolved Problems in Group Theory, 18th ed., Institute of Mathematics, Russian Academy of Sciences, Siberrian Division, Novosibirsk, arXiv:1401.0300v3 [math. GR] (2014).


\bibitem{15}R. Shen, G. Chen, C. Wu
On groups with the second largest value of the sum of element orders
Commun. Algebra, 43 (2015), pp. 2618-2631.
\bibitem{16} M. T\u{a}rn\u{a}uceanu,
Detecting structural properties of finite groups by the sum of element orders
Isr. J. Math., 238 (2020), pp. 629-637.
 \bibitem{Lind} F. Schmidt, Richard Stong, John H. Lindsey,  Schmidt's Query Becomes Lindsey's Theorem,
The American Mathematical Monthly, Vol. 98, No. 10 (Dec., 1991), pp. 970-972.
\bibitem{97}
W. Shi, A new characterization of the sporadic simple groups, in: Group theory (Singapore, 1987),
pp. 531–540, de Gruyter, Berlin  (1989).

\bibitem{Zal}  R. Zemlin,  On a conjecture arising from a theorem of Frobenius , Doctoral dissertation, Ohio State University, 1954. 


 
\end{thebibliography}

\end{document}